\documentclass[final]{siamltex}

\title{IMRO: A proximal quasi-Newton method for solving $\ell_1$-regularized least square problem
\thanks{Supported
in part by a grant from the U.~S.~Air Force Office of Scientific Research and in part
by a Discovery Grant from the Natural Sciences and Engineering Research Council
(NSERC) of Canada.}}
\author{Sahar Karimi\thanks{Department of Combinatorics \& Optimization,
University of Waterloo, 200 University Ave.~W., Waterloo, ON, N2L 3G1,
Canada, {\tt s2karimi@uwaterloo.ca}.} \and 
Stephen Vavasis\thanks{Department of Combinatorics \& Optimization,
University of Waterloo, 200 University Ave.~W., Waterloo, ON, N2L 3G1,
Canada, {\tt vavasis@uwaterloo.ca.}}}

\usepackage{hyperref,enumerate,enumitem, amsfonts, graphicx, amsmath,enumitem,multirow,epsfig,url,
nccmath,subfigure,multirow,threeparttable,
rotating,longtable,color}
\usepackage[normalem]{ulem}

\begin{document}

%%%%%%%
\newtheorem{thm}{Theorem}
\newtheorem{lem}{Lemma}
\newtheorem{cor}{Corollary}
\newtheorem{res}{Result}
\newtheorem{clm}{Claim}
\newtheorem{prop}{Property}
\newtheorem{prps}{Proposition}
\newtheorem{dfn}{Definition}
\newtheorem{alg}{Algorithm}
%%%%%%%%
\newcommand \beq {\begin{equation}}
\newcommand \eeq {\end{equation}}
\newcommand \beqn {\begin{equation*}}
\newcommand \eeqn {\end{equation*}}
\newcommand \bseq {\begin{subequations}}
\newcommand \eseq {\end{subequations}}
\newcommand \ba{\begin{array}}
\newcommand \ea{\end{array}}
\newcommand \beqa {\begin{eqnarray}}
\newcommand \eeqa {\end{eqnarray}}
\newcommand \beqan {\begin{eqnarray*}}
\newcommand \eeqan {\end{eqnarray*}}
\newcommand \bc{\begin{cases}}
\newcommand \ec{\end{cases}}
\newcommand \bitm{\begin{itemize}}
\newcommand \eitm{\end{itemize}}
%%%%%%%%%
\renewcommand \r {\right}
\renewcommand \l {\left}
%%%%%%%%%
\newcommand \eps {\epsilon}
\newcommand \onevec{\mathbf 1}
\newcommand \argmin {\arg \min}
\newcommand \sgn {\text{sgn}}
\newcommand \prox {\text{Prox}}
\newcommand \proj {\text{Proj}}
\newcommand \spn {\text{Span}}
\newcommand \shrk {\mathcal S}
\newcommand \xs {x^*}
\newcommand \lra {\Leftrightarrow}

\maketitle

\begin{abstract}
We present a proximal quasi-Newton method in which the approximation of the Hessian has the special format of ``identity minus rank one'' (IMRO) in each iteration. The proposed structure enables us to effectively recover the proximal point. The algorithm is applied to $\ell_1$-regularized least square problem arising in many applications including sparse recovery in compressive sensing, machine learning and statistics. Our numerical experiment suggests that the proposed technique competes favourably with other state-of-the-art solvers for this class of problems.  We also provide a complexity analysis for variants of IMRO, showing that it matches known best bounds.
\end{abstract}


%%%%%%%
\section{Introduction}
\label{intro:SR}
Compressive sensing (CS) \cite{RBcs, Ccs,  CWintocs} refers to the idea of encoding a large sparse signal through a relatively small number of linear measurements. This approach is essentially applying a linear operator $A\in \mathbb R^{m \times n}$ to a signal $x\in \mathbb R^n$ and storing $\hat b= Ax$ instead. Naturally we want $\hat b\in \mathbb R^m$ to be of a smaller dimension than $x$; hence in practice $m \ll n$. The main question is how to decode $\hat b$ to recover signal $x$, i.e., finding the solution to the underdetermined system of linear equations 
\beq\label{syslineq} A x = \hat b. \eeq

Sparse recovery particularly aims at finding the sparsest solution to \eqref{syslineq}. The sparsest solution might be obtained by solving 
\beq \label{x0obj} 
\ba{ll}
\min & \|x\|_0\\
\mbox{s.t.} & Ax=\hat b,
\ea
\eeq
where $\|x\|_0$ corresponds to the number of nonzero entries of $x$. Problem \eqref{x0obj} is, however, NP-hard and difficult to solve in practice. Therefore the  following linear programming relaxation was suggested for recovering the sparse solution: 
\beq\label{x1obj}
\ba{ll}
\min & \|x\|_1\\
\mbox{s.t.} & Ax= \hat b.
\ea
\eeq
The theory of compressive sensing has been well established. Cand\`es, Tao, Donoho, and Romberg are among the pioneers of compressive sensing theory; see \cite{CRT,CT,donoho} and references therein. In fact, they have shown that under some conditions \eqref{x1obj} can recover the solution to \eqref{x0obj}. 

In the presence of the noise in computing and storing $\hat b$, the measurement is often $b=\hat b +\hat \eps$; hence it is customary to replace 
$Ax=\hat b$ with $\|Ax- b\| \le \eps$ in \eqref{x1obj}, where $\eps$ is an estimated upper bound on the noise. The resulting problem is 
\beq\label{BP}
\ba{ll}
\min & \|x\|_1\\
\mbox{s.t.} & \|Ax - b\|\le \eps.
\ea
\quad(BP_\eps)
\eeq
Problem \eqref{x1obj} is usually referred to as ``Basis Pursuit" $(BP)$ problem, while $BP_\eps$ refers to its least square constrained variant, \eqref{BP}.

Other common problems in sparse recovery are
\beq\label{BPDN}
\ba{ll}
\min &\frac{1}{2}\|Ax-b\|^2+  \lambda \| x\|_1,
\ea
\quad \  (BPDN) 
\eeq
and
\beq \label{LASSO}
\ba{ll}
\min& \|Ax-b\|\\
\mbox{s.t.} & \|x\|_1\le \tau.
\ea
\hspace{1in} (LASSO) 
\eeq
In the literature of compressive sensing, \eqref{BPDN} is often called ``Basis Pursuit Denoising Problem" (BPDN) or $\ell_1$-regularized least square problem, and \eqref{LASSO} goes by the name of ``LASSO" (Least Absolute Shrinkage and Selection Operator). 
It is possible to show that formulations \eqref{BP}, \eqref{BPDN} and \eqref{LASSO} attain the same optimizer provided that certain relationship holds between $\epsilon$, $\lambda$ and $\tau$. However, there is no simple manner to compute this relationship without already knowing the optimal solution.
The algorithms proposed in this work are tailored for solving the BPDN problem, i.e. formulation \eqref{BPDN}. 

In what follows we first review some of the notation used in this paper followed by a brief review on some of the related  techniques to our method. IMRO is presented in Section \ref{sec:imro}. The convergence of IMRO is established in Section \ref{imro:cnvg}. In Section \ref{sec:fimro} we present the accelerated variant of IMRO. Our computational experiment is presented in Section \ref{ch:imroNR}. Finally we conclude our discussion in Section \ref{sec:conclusion}.

%%%%%%%
\section{Notation}
We work with real Euclidean vector space $\mathbb R^n$ equipped with the inner product $\langle x,y\rangle= \sum_{i=1}^n x_iy_i$. A linear operator (matrix) is defined as $A: \mathbb R^n \rightarrow \mathbb R^m$. Its adjoint is denoted by $A^t$, and we have $\langle Ax,y\rangle= \langle x,A^ty\rangle$.
Matrices, vectors and constants are denoted by upper-case, lower-case and Greek alphabet, respectively. The identity matrix is denoted by $\mathbf I$. We reserve the notation of $\|.\|$ for Euclidean norm; all other norms are denoted with the proper index. 

The notation $\mathbf C_L$ stands for the class of continuously differentiable convex functions with Lipschitz continuous gradient, where $L$ is the Lipschitz constant. We refer to functions of the form 
$$F(x):=f(x) + p(x)$$
as composite functions, where $f(x)\in \mathbf C_L$ and $p(x)$ is a convex, possibly nonsmooth, function. 
Note that BPDN problem is an example of minimization of a composite function. 

The proximal operator refers to 
$$\prox_p(y) = \arg\min_x \l\{ \frac{1}{2} \|x-y\|^2+ p(x) \r\}.$$
The (soft) shrinkage (also called thresholding) operator is denoted by $\mathcal S$, and defined as 
$$\shrk_\lambda (y) = \argmin_x\l\{ \frac{1}{2} \|x-y\|^2+ \lambda \|x\|_1\r\}.$$ 
Thus, $\shrk_\lambda\equiv\prox_p$ for the choice $p(x)=\lambda\|x\|_1$.
One may easily check that 
\beq%\label{shrkval}
\shrk_\lambda(y)_i=
\bc
y_i - \lambda & \mbox{if $y_i\ge \lambda$},\\
0 & \mbox{if $|y_i|\le \lambda$}, \\
y_i+\lambda & \mbox{if $y_i\le -\lambda$},
\ec
\eeq
which is equivalent to 
\beq%\label{shrkval2}
 \shrk_\lambda(y) = \sgn(y)\odot \max\l\{ |y|-\lambda, 0\r\},
 \eeq
where $\odot$ denotes the entry-wise or Hadamard product.

The scaled norm associated with a positive definite matrix $H\succ 0$ is defined as 
\beq %\label{scalednormeq}
\|x\|_H= \sqrt{x^tHx}. 
\eeq
It is not too difficult to see that the scaled norm satisfies all the axioms of a norm. Moreover, let the scaled proximal mapping (or operator) associated with positive definite matrix $H$ be defined as 
\beq%\label{sclproxmap}
\prox_p^H(y)= \argmin_x \l\{ \frac{1}{2}\| x-y\|_H^2 + p(x) \r\}.
\eeq

%%%%%%%
\section{Related Work}
Algorithms that rely solely on the function value and the gradient of each iterate are referred to as first-order methods. Due to the large size of the problems arising in compressive sensing, first-order methods are more desirable in sparse recovery. 
 There are numerous gradient-based first-order algorithms  
proposed for sparse recovery, see for example \cite{NESTAp,GPSRp,TwISTp1,FPCp2,VFried:PP,OG:bregman1,FN:EMalg,SpaRSAp,FMGenProx}. 

In \cite{VFried:PP,VFried:SPGL1} an efficient root finding procedure has been employed for finding the solution of $BP_\eps$ through solving a sequence of $LASSO$ problems. In other words, a sequence of $LASSO$ problems for different values of $\tau$ is solved using a spectral projected gradient method \cite{SpecProjGrad1}; and as $\tau \rightarrow \tau^\ast$, the solution of the $LASSO$ problem coincides with the solution of $BP_\eps$. In \cite{OG:bregman1}, the solution of $BP$ problem is recovered through solving a sequence of $BPDN$ problems with an updated observation vector $b$. GPSR \cite{GPSRp} is a gradient projection technique for solving the bound constrained QP reformulation of $BPDN$. 

Many other state-of-the-art algorithms in compressive sensing are inspired by iterative thresholding/shrinkage idea \cite{CPISTA1,CPISTA2,CWISTA3}.   
ISTA (iterative shrinkage thresholding algorithm) is an extension of the steepest descent idea to composite functions using the thresholding operator.
Recall that in the steepest descent method, the general format of the generated sequence is
\beq x^{k+1} = x^k - \alpha^k \nabla f (x^k),\eeq
which might be considered as the solution to the following quadratic 
approximation of $f$:
\beq\label{sdmodel} 
x^{k+1} = \argmin_x \l\{ f(x^k)+ \langle x-x^k, \nabla f(x^k) \rangle+ \frac{1}{2\alpha^k}\|x-x^k\|^2 \r\}. 
 \eeq

%%%%%%%
Consider the problem of minimizing a composite function, i.e.,
\beq \label{mincompfunintro}
\min_x \ \ F(x)= f(x) + p(x).
\eeq
Using the same approximation model as in \eqref{sdmodel} for $f(x)$, we can get the following iterative scheme:
\beq\label{mincompfuna1}
 x^{k+1} = \argmin_x \l\{ f(x^k)+ \langle x-x^k, \nabla f(x^k) \rangle+ \frac{1}{2\alpha^k}\|x-x^k\|^2 +p(x)\r\}. 
\eeq

Shuffling the linear and quadratic terms and ignoring the constants in \eqref{mincompfuna1}, it can equivalently be written as
\beq\label{mincompfuna2}
 x^{k+1} = \argmin_x \l\{ \frac{1}{2\alpha^k}\|x-\l(x^k- \alpha^k \nabla f(x^k)\r)\|^2 +p(x)\r\}. 
\eeq
Using the notion of $\prox$ operator we conclude that
\beq\label{prox1}
x^{k+1}=\prox_{\alpha^k p}\l(x^k- \alpha^k \nabla f(x^k)\r).
\eeq
The iterative scheme of \eqref{prox1} is the ``generalized gradient method" or ``proximal gradient method". Note that it actually coincides with steepest descent method in the absence of $p(x)$. 
 It is also sometimes called ``Forward-backward Splitting Method" \cite{CWISTA3, FMGenProx,FBsplitSysIden}. 
Finding the proximal point may not be a trivial task in general, but for solving $BPDN$ it can be computed efficiently because the $\ell_1$-norm is separable.  

The algorithm that goes by the name of ISTA in the literature of sparse recovery refers to a proximal gradient method for composite functions in which $p(x)=\lambda \|x\|_1$. The general form of ISTA is 
\beq\label{istaGF}
x^{k+1}=\shrk_{\lambda \alpha^k} \l(x^k- \alpha^k \nabla f(x^k)\r).
\eeq

Each iteration of ISTA can be computed efficiently; it, however, could suffer from slow rate of convergence. In general, it has sublinear (i.e. $\mathcal O(k^{-1})$) rate of convergence \cite{BT:fista, BL:linCnvgISTA}. In  \cite{BL:linCnvgISTA}, it has been shown that 
generalized gradient algorithms achieve linear convergence.

%%%%%%%
FISTA (fast ISTA) \cite{BT:fista} is the accelerated variant of ISTA that was built upon the Nesterov's idea \cite{nesterov1983,nesterov2005}. Each iteration of FISTA has the following format:
\beqa
x^{k+1}& = &\shrk_{\lambda \alpha^k } \l(y^k-\alpha^k \nabla f(y^k)\r), \label{fistaGFe1}\\ 
t^{k+1}& = &  \frac{1+ \sqrt{1+4{t^k}^2}}{2},\label{fistaGFe2}\\
y^{k+1}& = &  x^k + \l(\frac{t^k-1}{t^{k+1}}\r)\l( x^{k+1} - x^k\r).\label{fistaGFe3}
\eeqa
FISTA is not restricted to BPDN problem, and it was originally proposed for minimizing a general composite function. Replacing \eqref{fistaGFe1} with 
\beq\label{fistaGFge1}
 x^{k+1}=\prox_{ \alpha^k p} \l(y^k-\alpha^k \nabla f(y^k)\r) 
\eeq
would generalize FISTA to an algorithm well-suited for minimizing any composite function. 
The Nesterov's accelerated proximal gradient algorithm has been adopted for solving $BP_\eps$ in \cite{NESTAp}, and  for solving the $LASSO$ problem in \cite{ParNes}. 

The alternating direction method (ADM) is also a technique that can be applied to $BPDN$, see \cite{Z:ADM, Z:fastADM} and references therein. It is suited for minimizing the summation of (separable) convex functions, say $f(x)+p(y)$, over a linear set of constraints. The augmented Lagrangian technique then solves for $x$ and $y$ alternately while fixing the other variable. The alternating linearization method (ALM) \cite{KG:ALM} also applies to minimizing composite functions. In \eqref{mincompfuna1}, we linearize $f$ at every iteration to build the quadratic approximation model; in ALM a similar model based on $p$ is also minimized at every iteration. Nesterov's accelerated technique has also been adopted, and the resulting algorithm is called FALM, for fast ALM. 

In order to incorporate more information about the function without trading off the efficiency of the algorithms, Newton/quasi-Newton proximal methods \cite{BeckerPNM,ProxNMcompfun} have attracted researchers quite recently. Most of the previous extensions on quasi-Newton methods are suited either for nonsmooth problems \cite{LO:NSO,QNinML}, or for constrained problems with simple enough constraints \cite{QN4BC, QN4BC2, Fried:QN4Cnst}. The technique mostly related to ours is the quasi-Newton method of Becker and Fadili \cite{BeckerPNM} described in section \ref{sec:imro}.

The proximal quasi-Newton method is obtained by replacing the diagonal matrix $\frac{1}{\alpha^k}\mathbf I$ in the quadratic term of \eqref{mincompfuna1} with a suitable positive definite matrix. In other words, define $m_{H^k}(x,x^k)$ as 
\beq\label{quadmodelH}
m_{H^k}(x,x^k)= f(x^k)+ \langle x-x^k, \nabla f(x^k) \rangle+ \frac{1}{2}\l(x-x^k\r)^tH^k\l(x-x^k\r),
\eeq
where $H^k\succ0$, and solve 
\beq\label{minquadmodelH}
\min_x \ m_{H^k}(x,x^k) + p(x),
\eeq
at each iteration. 
 Ignoring the constant terms and using the definition of scaled norm, we can rewrite \eqref{quadmodelH} as 
\beq\label{quadmodelHn}
m_{H^k}(x,x^k)=\frac{1}{2}\|x-\l(x^k-({H^k})^{-1}\nabla f(x^k)\r)\|^2_{H^k}, 
\eeq
and define the proximal quasi-Newton algorithm as
\beq\label{PQNMe1}
x^{k+1}=\prox^{H^k}_p\l(x^k-({H^k})^{-1}\nabla f(x^k)\r).
\eeq
%%%%%%%%%%%%%%%
%%%%%%%%%%%%%%%
\section{IMRO Algorithm} \label{sec:imro}
We present a practical variant of proximal quasi-Newton methods for solving BPDN problem in this section. Recall that the BPDN problem is 
\beq\label{imro:BPDNdef}
\min_x \ \ \tilde F(x):= \ \frac{1}{2} \| Ax-b\|^2 + \lambda \|x\|_1.
\eeq
Let us denote the quadratic part of $\tilde F(x)$ with $\tilde f(x)$, and the $\ell_1$-regularization term with $\tilde p(x)$. We note that $\tilde f\in \mathbf{C}_L$ with $L=\Vert A\Vert^2$.
Applying the proximal quasi-Newton scheme of  \eqref{minquadmodelH} to BPDN we get that 
\beq\label{findxQmodel}
\ba{ll}
x^{k+1}= \displaystyle{\arg\min_x}\ \ \frac{1}{2}\|x-\l(x^k-({H^k})^{-1}\nabla \tilde f(x^k)\r)\|^2_{H^k} + \lambda\|x\|_1.
\ea
\eeq
In our proposed proximal quasi-Newton scheme $H^k$ has the following format:
\beq
H^k= \sigma^k \mathbf I - u^k({u^k})^t. 
\eeq
Note that $H^k\succ 0$ provided $\sigma^k>\| u^k\|^2$.
In fact the term ``IMRO" stands for ``identity minus rank one," which is the proposed format for matrix $H^k$. We will see shortly that one of the advantages of IMRO is the efficiency of computing $x^{k+1}$. In \cite{BeckerPNM}, Becker and  Fadili suggest a proximal quasi-Newton method in which $H^k$ is a positive definite diagonal plus a rank one matrix. The methodology that we develop for selecting $\sigma$ and $u$ presented in Sections \ref{sec:imro1d} and \ref{sec:imro2d} does not seem to extend to the case of identity (or positive definite diagonal) plus rank one according to our analysis, but this question may need future investigation. Computing $x^{k+1}$ in IMRO, however, is very similar to what is proposed in \cite{BeckerPNM} and our discussion in Section \ref{imro:compnextitr} can be applied to the identity plus rank one Hessian approximation. 

In the remainder of this section, we first describe how we may find $x^{k+1}$ utilizing the special structure of $H^k$ in IMRO. Our discussion is then followed by two different variants of IMRO and their properties.

\subsection{Computing the Proximal Point in IMRO}\label{imro:compnextitr}
In this section we explain how we can attain the solution of \eqref{findxQmodel}, denoted by $x^+$ here, in linearithmic time, i.e., $O(n\log n)$.
Note that optimality conditions for \eqref{findxQmodel} imply that 
\beq \label{imro:findx:optcondH}
H^k\l(x^+-(x^k-({H^k})^{-1}\nabla \tilde f^k)\r) + \lambda \xi^+ =0,
\eeq 
where $\xi^+ \in \partial(\|x^{+}\|_1)$ is a subgradient of $\tilde p(x)$ at $x^+$. Let us denote $x^k-({H^k})^{-1}\nabla \tilde f^k$ by $\bar x^k$. Since in IMRO $H^k=\sigma^k \mathbf I - u^k(u^k)^t$, $({H^k})^{-1}$ might be computed in closed form:
\beq
(\sigma^k \mathbf I - u^k(u^k)^t) ^{-1}= \frac{1}{\sigma^k} \mathbf I - \frac{1}{\sigma^k (\|u^k\|^2-\sigma^k)}u^k(u^k)^t, 
\eeq
so we are able to calculate $\bar x^k$ easily. Condition \eqref{imro:findx:optcondH} may now be restated as 
\beq \label{imro:findx:optcond}
(\sigma^k \mathbf I - u^k(u^k)^t) (x^+-\bar x^k)+ \lambda \xi^+ =0.
\eeq
Recall that $\xi^+_i=1$ if $x_i^+>0$, $\xi_i^+=-1$ if $x_i^+<0$, and $\xi_i^+\in[-1,1]$ if
$x_i^+=0$.  In the latter case, we have the freedom to select $\xi_i^+$ as any point in 
$[-1,1]$ in order to make (\ref{imro:findx:optcond}) hold.

Let $\mu$ (to be found) be a scalar equal to $\frac{(u^k)^t(x^+-\bar x^k)}{\sigma^k}$. Then equation \eqref{imro:findx:optcond} reduces to 
\beq\label{imro:findx:optcondmu}
x^+-\bar x^k - u^k \mu + \frac{\lambda}{\sigma^k} \xi^+ =0;
\eeq
thus
\beq \label{imro:findx:xform}
x^{+}_i= \l\{
  \ba{ll}
	\bar x^k_i+ u^k_i\mu -\lambda/\sigma^k & \text{if\ \ } x^k_i+ u^k_i\mu -\lambda/\sigma^k>0,\\
		\bar x^k_i+ u^k_i\mu +\lambda/\sigma^k & \text{if\ \ } x^k_i+ u^k_i\mu +\lambda/\sigma^k<0,\\
    0 & \text{otherwise.}
  \ea
\r.
\eeq
Using \eqref{imro:findx:xform} and sign of $u^k_i$, we may find the proper interval for $\mu$ so that the mentioned equations for $x_i^+$ holds true; in other words:
\beqa
\label{imro:findx:breakpnt1}
x_i^+>0 & 
\Rightarrow& \bar x^k_i +u^k_i\mu -\frac{\lambda}{\sigma^k} >0  \Rightarrow \ba{l} \mu> \frac{\lambda/\sigma^k-\bar x^k_i}{u^k_i} \text{ if } u^k_i>0,\\ \mu< \frac{\lambda/\sigma^k-\bar x^k_i}{u^k_i} \text{ if } u^k_i<0,\ea \\ 
& \notag \\
x_i^+<0 &
\Rightarrow & \bar x^k_i +u^k_i\mu +\frac{\lambda}{\sigma^k} <0   \Rightarrow \ba{l} \mu< \frac{-\lambda/\sigma^k -\bar x_i^k}{u^k_i} \text{ if } u^k_i>0,\\  \mu> \frac{-\lambda/\sigma^k-\bar x_i^k}{u^k_i} \text{ if } u_i^k<0,\ea \label{imro:findx:breakpnt2}\\
x_i^+ = 0 & \Rightarrow &\displaystyle  \xi_i^+ = \frac{(\bar x_i^k + u^k_i\mu)\sigma^k}{\lambda}\label{imro:findx:breakpnt3}.
\eeqa
Note that by definition of $\mu$, we have
\beq\label{imro:findx:mueq1}
 (u^k)^t x^+ - \mu \sigma^k  = (u^k)^t \bar x^k.
\eeq
Searching over all the breakpoints mentioned in \eqref{imro:findx:breakpnt1} and \eqref{imro:findx:breakpnt2}
(i.e. $\frac{\lambda/\sigma^k-\bar x_i^k}{u^k_i}$ and $\frac{-\lambda/\sigma^k-\bar x_i^k}{u_i^k}$), enables us to find  the proper value of $\mu$ for which \eqref{imro:findx:mueq1} holds. 
By taking the inner product of both sides of \eqref{imro:findx:optcondmu}
with $u^k$, we obtain
\beq 
(u^k)^t x^+ = (u^k)^t\left(\bar x^k -\frac{\lambda}{\sigma^k}\xi^+\right) + (u^k)^tu^k \mu, 
\eeq
hence equation \eqref{imro:findx:mueq1} has the equivalent form of 
\beq \label{imro:findx:mueq2}
(\text{lhs})\ \  (u^k)^t\left(\bar x^k - \frac{\lambda}{\sigma^k}\xi^+ \right) + \l((u^k)^tu^k-\sigma^k\r) \mu = (u^k)^t \bar x^k \ \ (\text{rhs}).
\eeq
Note that both terms on the left-hand side are functions of $\mu$, the first term
via the implicit dependence of $\xi^+_i$ on $\mu$ whenever $x_i^+=0$,
while the second term explicitly depends on $\mu$. 
Thus,
we see that the left-hand side of \eqref{imro:findx:mueq2} is a piecewise linear
continuous function of $\mu$, where the pieces are given by intervals between the
above-mentioned breakpoints.  Furthermore, the slope is always nonpositive
because the second term contributes $(u^k)^tu^k-\sigma^k$ to the slope, a negative
number, while the $i$th contribution from the first term is either 0
(when $x_i^+\ne 0$) or $-({u_i^k})^2$ (when $x_i^+=0$). 

This monotonicity allows us to find the correct $\mu$ solving
\eqref{imro:findx:mueq2}.
To find $\mu$, we sort all the breakpoints (a vector of size $2n$); we start with an initial value of $\mu$ small enough such that lhs$>$rhs; we then increment the value of $\mu$ over the sorted breakpoints until we reach the desired interval $[\mu_l, \mu_u]$ such that lhs$_{\mu_l} >$ rhs and lhs$_{\mu_u} <$ rhs, or the value of $\mu^\ast$ for which lhs$_{\mu^\ast} =$ rhs. In the case that we reach the interval, a simple interpolation solves \eqref{imro:findx:mueq2}.
Note that we may efficiently update the lhs when reaching a breakpoint, since only one of $x^+_i$'s changes sign for each breakpoint. The following chart visualizes how the search process is actually carried out:

$$ u^k_i>0: \ \ \underbrace {-----\mid}_{x_i^+<0} \frac{-\frac{\lambda}{\sigma^k}-\bar x_i^k}{u^k_i}\underbrace{\mid -----\mid}_{x_i^+=0}\frac{\frac{\lambda}{\sigma^k}-\bar x_i^k}{u^k_i} \underbrace{\mid-----}_{x_i^+>0} $$

$$u_i^k<0: \ \ \underbrace {-----\mid}_{x_i^+>0} \frac{\frac{\lambda}{\sigma^k}-\bar x_i^k}{u^k_i}\underbrace{\mid -----\mid}_{x_i^+=0}\frac{-\frac{\lambda}{\sigma^k}-\bar x_i^k}{u^k_i} \underbrace{\mid-----}_{x_i^+<0} $$

The algorithm below summarizes all we said above for finding $\mu$ and $x^{k+1}$. The presented pseudocode is in MATLAB notation. ``slp" in the following algorithm stands for the slope of lhs in \eqref{imro:findx:mueq2} of the current piece (i.e., the derivative with respect
to $\mu$).

\begin{alg}\label{algfindx}
\begin{tabbing}
++\= +++ \= +++\= +++\= +++\=\kill
Input: $\sigma^k$, $u^k$, $\bar x^k$, and $\lambda$\\
 \verb+ slp-Update Subroutine+:\\
\> Let $i=|\bar a(j,2)|$\\
\> {\bf if} $\bar a(j,2)<0$\\
\>\> {\bf if} $u_i^k<0$\\
\>\>\> slp= slp + $(u_i^k)^2$\\
\>\> {\bf else} \\
\>\>\> slp= slp - $(u_i^k)^2$\\
\> {\bf else}\\
\>\> {\bf if} $u_i^k<0$\\
\>\>\> slp= slp - $(u_i^k)^2$\\
\>\> {\bf else} \\
\>\>\> slp= slp + $(u_i^k)^2$\\

\verb+ Main Procedure+\\
\> Let $\mathcal I = \{i: u^k_i\neq 0\}$\\
\> Form $a\in \mathbb R^{2|\mathcal I| \times 2}$ such that $a(i,:)=[\frac{\frac{\lambda}{\sigma^k}-\bar x_i^k}{u^k_i}, +i]$ and $a(|\mathcal I|+i,:)=[\frac{-\frac{\lambda}{\sigma^k}-\bar x_i^k}{u^k_i}, -i]$ \\
\> Let $\bar a:=$ sorted ``$a$" on first column\\
\> Let rhs$:=(u^k)^t\bar x^k$\\
\> Choose $\mu<\bar a(1,1)$ such that lhs$:=(u^k)^tx^\mu-\mu \sigma^k>rhs$,\\
\>\> where $x^\mu$ is derived by \eqref{imro:findx:xform}\\
\>\> slp$=(u_{\mathcal I^0}^k)^tu_{\mathcal I^0}^k-\sigma^k$, where $\mathcal I^0= \{ i: x^\mu_i \neq 0\}$\\
\> {\bf for} $j=1, 2,\ldots 2|\mathcal I|$\\
\>\> Let $\mu^+=\bar a(j,1)$\\
\>\> lhs$^+=$ lhs + slp($\mu^+ - \mu$)\\
\>\> Update $slp$ using \verb+ slp-Update Subroutine+ ($\bar a(j,:)$, slp)\\
\>\> {\bf if} lhs$^+ \le$ rhs\\
\>\>\> $\mu^\ast= \frac{(rhs -lhs^+)\mu}{lhs - lhs^+} + \frac{(lhs - rhs)\mu^+}{lhs - lhs^+} $, \\
\>\>\> Derive $x^{k+1}$ by \eqref{imro:findx:xform} for $\mu=\mu^\ast$.\\
\>\>\> return\\
\>\> $\mu=\mu^+$ and lhs = lhs$^+$\\
\end{tabbing}
\end{alg}

The computation of $x^{k+1}$ can actually be done in linear time, i.e., $O(n)$  (rather than $O(n \log n)$).  The linear-time algorithm for finding $\mu$
is based on finding the median of an unsorted array of size $n$ in linear time.  So after computing the $2n$ breakpoints, we can find the median of the breakpoints and calculate the lhs and rhs of \eqref{imro:findx:mueq2} in $O(n)$. If the lhs$\ge$rhs, then we can discard all the breakpoints below the median. Likewise, if lhs$\le$rhs we can drop all the values above the median. This step can also be done in $O(n)$, and reduces the  size of the problem to $\frac{n}{2}$. The same procedure can be applied to the remaining breakpoints until we reach the desired interval for $\mu$ (an interval $[\mu_l, \ \mu_u]$ such that $lhs_{\mu_l} \ge rhs$ and $lhs_{\mu_u} \le rhs$). Thus, the total running time is of the form $O(n)+O(\frac{n}{2}) + O(\frac{n}{4})+\cdots$ which is $O(n)$.

Two variants of IMRO are proposed in this paper. The difference between these two variants lies in the derivation of $\sigma^k$ and $u^k$. We refer to these variants as IMRO-1D for IMRO on a one-dimensional subspace, and IMRO-2D for IMRO on a two-dimensional subspace.

\subsection{IMRO-1D}\label{sec:imro1d}
In IMRO-1D, we find $\sigma^k$ and $u^k$ such that the approximation model $m_{H^k}(x,x^k)$ equals $\tilde f(x)$ on a one-dimensional affine space $x^k+\alpha v^k$, where $v^k$ is a direction of our choice. We discuss the choice of $v^k$ later in this paper. Moreover, we require $m_{H^k}(x,x^k)$ to be an upper approximation for $\tilde f(x)$. The latter property has some theoretical benefits in the convergence of the algorithm as we shall see in Section \ref{imro:cnvg}. The formal statement of these imposed constraints is 
\beqa
m_{H^k}(x,x^k)&=& \tilde f(x) \ \ \text{whenever} \ \ x\in x^k+ \spn\l\{v^k\r\}, \label{imro1d:eqonv:eq1}\\
m_{H^k}(x,x^k)&\ge & \tilde f(x) \ \ \forall x \in \mathbb R^n,\label{imro1d:modfun:inqe1}
\eeqa
for some nonzero vector $v^k$ to be determined later.
Using \eqref{quadmodelH}, we deduce that  \eqref{imro1d:eqonv:eq1} is equal to
\beq\label{imro1d:eqonv:eq2}
\frac{1}{2}({v^k})^tH^kv^k=\frac{1}{2}(v^k)^tA^tAv^k,
\eeq
and condition \eqref{imro1d:modfun:inqe1} implies that
\beq \label{imro1d:modfun:ineq2}
\frac{1}{2}(x-x^k)^tH^k(x-x^k)\ge \frac{1}{2}(x-x^k)A^tA(x-x^k).
\eeq

Obviously \eqref{imro1d:modfun:ineq2} holds if and only if $H^k\succeq A^tA$. By \eqref{imro1d:eqonv:eq2} and \eqref{imro1d:modfun:ineq2}, the required conditions on $H^k$ boils down to 
\beqa\label{imro1d:Hreq1}
(v^k)^t ( H^k &-& A^tA) v^k =0,\\
H^k&\succeq & A^tA.\label{imro1d:Hreq2}
\eeqa

In the rest of this subsection we show how we can compute $\sigma^k$ and $u^k$ such that the above conditions are satisfied.
%%%%%%
\subsubsection{Finding $\sigma^k$ and $u^k$ in IMRO-1D}
Conditions \eqref{imro1d:Hreq1} and \eqref{imro1d:Hreq2} imply that
\beq
 v^k\in \mathcal N(H^k-A^tA).
\eeq
Without loss of generality, we assume that $v^k$ is normalized, i.e., $\|v^k\|=1$. The following lemma gives us the formula for $\sigma^k$ and $u^k$ in IMRO-1D. 

\begin{lem}
\eqref{imro1d:Hreq1} and \eqref{imro1d:Hreq2} are satisfied for 
\beq\label{imro1dsig}
\sigma^k=\|A\|^2,
\eeq
and 
\beq \label{imro1du} 
u^k= \begin{cases}
\frac{\sigma^k v^k -A^tAv^k}{\sqrt{\sigma^k - \|Av^k\|^2}}& \ \ \text{ if $v^k$ is not a dominant singular vector of $A$}, \\
0 & \ \ \text{otherwise}.
\end{cases}
\eeq

 \end{lem}
\begin{proof}
Note that $\|A\|^2= \lambda_{\max}(A^tA)=\delta^2_{\max}(A)$, where $\lambda_{\max}$ and $\delta_{\max}$ stand for the maximum eigenvalue and maximum singular value, respectively. 
Let us first consider the case where $v^k$ is a dominant singular vector of $A$. In this case $H^k=\sigma^k \mathbf I = \|A\|^2 \mathbf I \succeq A^tA$ and $(\sigma^k \mathbf I - A^tA)v^k=0$, so both requirements hold.

Suppose $v^k$ is not a dominant singular vector of $A$. Then the denominator in the formula for $u^k$ is positive and $u^k$ is defined. We, therefore, have
\beqan
(H^k-A^tA)v^k& =& \sigma^k v^k - \l((u^k)^tv^k\r)u^k - A^tAv^k \\
&=& \sigma^k v^k - \l(\sqrt{\sigma^k - \|Av^k\|^2}\r) \frac{\sigma^k v^k -A^tAv^k}{\sqrt{\sigma^k - \|Av^k\|^2}} - A^tAv^k=0,
\eeqan
which concludes equality \eqref{imro1d:Hreq1}. It remains to show \eqref{imro1d:Hreq2}, that is $x^t(\sigma^k \mathbf I - u^k(u^k)^t)x \ge x^tA^tAx$ for all $x\in \mathbb R^n$. Equivalently, we will show that for all $x\in \mathbb R^n$ such that $\|x\|=1$, we have $\sigma^k \ge x^tA^tAx + \l((u^k)^tx\r)^2$,
i.e.,
$$\sigma^k \ge \sup_{\|x\|=1} \l\| \l(\ba{c} A\\ (u^k)^t \ea\r)x\r\|^2=\l\| \l(\ba{c} A\\ (u^k)^t \ea\r)\r\|^2.$$
In fact, we prove that $\sigma^k =  \l\| \l( \ba{c}A \\ (u^k)^t \ea \r) \r\|^2$.
Clearly 
\beqn
\l\|\l(\ba{c}A\\ (u^k)^t\ea\r)\r\| \ge \|A\|,
\eeqn 
because 
\beqn
\l\|\l(\ba{c}A\\ (u^k)^t\ea\r)x\r\|=\l\|\l(\ba{c}Ax\\ (u^k)^tx\ea\r)\r\| \ge \|Ax\| \ \ \forall x\in \mathbb R^n.
\eeqn

It remains to show that $\l\|\l(\ba{c}A\\ (u^k)^t\ea\r)\r\| \le \|A\|$. By the value of $\sigma^k$, we have $\sigma^k \mathbf I -A^tA\succeq 0$, so we can define $B$ such that  $BB^t=\sigma^k \mathbf I - A^tA$. Note that 
\beqa \label{e3.2.1}
x^t \l( A^t \ \ u^k\r) \l(\ba{c}A\\(u^k)^t\ea\r) x = x^t A^t A x + \l((u^k)^tx\r)^2 &=& x^tA^tAx + \frac{\l(x^t\l(\sigma^k \mathbf I- A^tA\r)v^k\r)^2}{\sigma^k - \|Av^k\|^2}\nonumber \\
&=& x^tA^tAx+ \frac{\l(x^t\l(\sigma^k \mathbf I- A^tA\r)v^k\r)^2}{(v^k)^t\l(\sigma^k - A^tA\r)v^k}\nonumber\\
&=&x^tA^tAx+ \frac{\l(x^tBB^tv^k\r)^2}{(v^k)^tBB^tv^k}\nonumber\\
&\le&x^tA^tAx + x^tBB^tx= \sigma^k x^tx,
\eeqa
where the last inequality is ensured by Cauchy-Schwarz inequality, i.e.,
\beqn
\l(x^tBB^tv^k\r)^2 \le \|B^tx\|^2\|B^tv^k\|^2= \l(x^tBB^x\r) \l((v^k)^tBB^tv^k\r).
\eeqn
Combining the definition of induced matrix norms and the result obtained in \eqref{e3.2.1} we get
\beqn
\l\|\l(\ba{c}A\\ (u^k)^t\ea\r)\r\|= \sup_{x: \|x\|=1} x^t \l( A^t \ \ u^k\r) \l(\ba{c}A\\(u^k)^t\ea\r) x \le \sup_{x:\|x\|=1} \sigma^k x^tx = \sigma^k =\|A\|
\eeqn
which yields the result we wanted to show.
\end{proof}

%%%
\subsection{IMRO-2D}\label{sec:imro2d}
In IMRO-2D the quadratic model $m_{H^k}(x,x^k)$ matches the function on the two-dimensional space of $x^k + \spn\l\{\nabla \tilde f^k, d^k\r\}$, where $d^k=x^k-x^{k-1}$. Without loss of generality, let us assume that $\nabla \tilde f^k$ and $d^k$ are normalized. 

The imposed condition for IMRO-2D requires
\beq \label{imro2d:cond1}
\tilde f(x^k) + \langle \nabla \tilde f^k, x-x^k\rangle + \frac{1}{2}(x-x^k)^t H^k(x-x^k) = \frac{1}{2}\|A\l(x^k+ (x-x^k)\r)\|^2, 
\eeq
for all $x\in \l\{x^k+ \spn\l\{\nabla \tilde f^k, d^k\r\} \r\}$, that is when $x-x^k= \alpha \nabla \tilde f^k + \beta d^k$. 

Condition \eqref{imro2d:cond1}, therefore, reduces to 
\beqa \label{imro2d:cond2}
\frac{1}{2}(x-x^k)^tH^k(x-x^k)&=&\frac{1}{2}(x-x^k)^tA^tA(x-x^k), \ \text{i.e.,} \nonumber\\
\frac{1}{2}(\alpha \nabla \tilde f^k + \beta d^k)^t H^k (\alpha \nabla \tilde f^k + \beta d^k) &=& \frac{1}{2} (\alpha \nabla \tilde f^k + \beta d^k)^t A^t A (\alpha \nabla \tilde f^k + \beta d^k) 
\eeqa
for all $\alpha, \beta \in \mathbb{R}$.
The fact that $H^k=\sigma^k \mathbf I - u^k(u^k)^t$ in IMRO enables us to find $\sigma^k, \tau^k,$ and $\rho^k$ such that \eqref{imro2d:cond2} is satisfied for $\sigma^k$ and $u^k=\tau^k \nabla \tilde f^k + \rho^k d^k$. This is the topic covered in the remainder of this subsection. 
%%%%
\subsubsection{Finding $\sigma^k$ and $u^k$ in IMRO-2D}
By \eqref{imro2d:cond2}, we need to solve
\beq\label{imro2d:sigucomp:modeqfun}
(\alpha \nabla \tilde f^k + \beta d^k)^t A^t A  (\alpha \nabla \tilde f^k + \beta d^k)= (\alpha \nabla \tilde f^k + \beta d^k)^t(\sigma^k I - u^k(u^k)^t) (\alpha \nabla \tilde f^k + \beta d^k),
\eeq
for $\sigma^k$ and $u^k$. We first derive $\sigma^k$, then using $\sigma^k$ we will compute vector $u^k$.

Let $S$ be the following matrix
\beq 
S=\l( \nabla \tilde f^k \ \ d^k\r)^t A^t A\l(\nabla \tilde f^k \ \ d^k\r)= \l(\ba{cc} ({\nabla \tilde f^k})^tA^tA\nabla \tilde f^k &({\nabla \tilde f^k})^tA^tAd^k\\ ({d^k})^tA^tA\nabla \tilde f^k & ({d^k})^tA^tAd^k \ea\r).
\eeq
Then \eqref{imro2d:sigucomp:modeqfun} imposes the following equations on $\sigma^k$ and $u^k$:
\beqa\label{imro2d:sig:eq1}
S_{11}&=& \sigma^k - ({\nabla \tilde f^k})^tu^k(u^k)^t\nabla \tilde f^k, \nonumber\\
S_{12}&=&\sigma^k ({\nabla \tilde f^k})^t{d^k} - ({\nabla \tilde f^k})^tu^k(u^k)^td^k,\\
S_{22}&=& \sigma^k - ({d^k})^tu^k(u^k)^td^k\nonumber.
\eeqa
Set $\eps^k$ as $(\nabla \tilde f^k)^t d^k$, an easily computable scalar. Then by the set of equations in \eqref{imro2d:sig:eq1} we have
\beqan
\det(S)&=& S_{11}S_{22} -S_{12}^2 \\
&=&(\sigma^k)^2\l(1-(\epsilon^k\r)^2) + \sigma^k \l( S_{11}-\sigma^k +S_{22}-\sigma^k + 2\epsilon^k (\sigma^k \epsilon^k- S_{12})\r).
\eeqan
Hence $\sigma^k$ can be calculated by solving the following quadratic equation
\beq \label{imro2d:sig:quad}
(\sigma^k)^2 \l(1-(\epsilon^k)^2\r) + \sigma^k ( -S_{11} -S_{22}+ 2\epsilon^k S_{12}) + \det(S)=0.
\eeq
Suppose $u^k=\tau^k \nabla \tilde f^k + \rho^k d^k$. Using \eqref{imro2d:sig:eq1}, we get
\beqa \label{imro2d:getu:eq1}
\sigma^k - S_{11}&=& (\tau^k)^2 + 2\epsilon^k \tau^k \rho^k  + (\epsilon^k)^2 (\rho^k)^2 = (\tau^k + \epsilon^k \rho^k)^2, \nonumber
\\
\epsilon^k \sigma^k - S_{12}&=& \epsilon^k (\tau^k)^2 + \tau^k \rho^k+ (\epsilon^k)^2 \tau^k \rho^k  + \epsilon^k (\rho^k)^2 \\
&=& (\tau^k + \epsilon^k \rho^k)(\epsilon^k \tau^k +  \rho^k),
\nonumber \\
 \sigma^k - S_{22}&=& (\epsilon^k)^2 (\tau^k)^2 + 2\epsilon^k \tau^k \rho^k  +  (\rho^k)^2 = (\epsilon^k \tau^k +  \rho^k)^2 \nonumber,
\eeqa
so $(\tau^k +  \epsilon^k\rho^k) = \sqrt{\sigma^k - S_{11}}$ and $(\epsilon^k \tau^k +  \rho^k) = \sqrt{\sigma^k - S_{22}}\  \sgn (\epsilon^k  \sigma^k - S_{12})$. Therefore $\tau^k$ and $\rho^k$ are the solutions of the following linear system:
\beq \label{imro2d:getu:linsys}
\l( \ba{cc} 1& \epsilon^k \\ \epsilon^k & 1 \ea \r) \l( \ba{c} \tau^k\\ \rho^k \ea \r) = \l( \ba{c} \sqrt{\sigma^k - S_{11}}\\  \sqrt{\sigma^k - S_{22}}\text{ sgn}(\epsilon^k  \sigma^k - S_{12})\ea \r).
\eeq

%%%%%%%%
\subsubsection{Validity of IMRO-2D}
In what follows, we show that IMRO-2D is a valid algorithm, namely $\exists \ \sigma^k, \tau^k, \rho^k \in \mathbb{R}$ that solve \eqref{imro2d:sig:quad} and \eqref{imro2d:getu:linsys}.

\begin{prop}\label{imro2d:prop1} Let $\eta_1=1-(\epsilon^k)^2$, $\eta_2= -S_{11} -S_{22}+ 2\epsilon^k S_{12}$, and $\eta_3= \det(S)$ be the coefficients of the quadratic equation \eqref{imro2d:sig:quad}. Then 
\beqa 
\eta_1&\ge&0,\\
 \eta_2&\le&0,\\
 \eta_3&\ge&0.
 \eeqa
 \end{prop}
\begin{proof}$ $
\bitm [noitemsep,nolistsep]
\item Note that $(\epsilon^k)^2 \le \epsilon^k \le 1$ because $\eps^k=(\nabla \tilde f^k)^td^k$ and $\|\nabla \tilde f^k\|=\|d^k\|=1$, therefore $\eta_1=1-(\epsilon^k)^2\ge0$. 
\item$\eta_3=\det(S) \text{ and } S\succeq 0 \Rightarrow \eta_3\ge 0$.
\item $(S_{11}-S_{22})^2\ge 0 \ \lra \ S_{11}^2+S_{22}^2+2S_{11}S_{22} \ge 4S_{11}S_{22}\ge 4S_{12}^2\ge 4(\eps^k)^2S_{12}^2$,\\
where the last two inequalities hold by $S\succeq 0$ and $(\eps^k)^2\le1$, respectively; therefore \\
$(S_{11}+S_{22})^2 \ge 4(\eps^k)^2 S_{12}^2\ \Rightarrow S_{11}+S_{22}\ge 2\epsilon^k S_{12} \Rightarrow \eta_2\le 0 $.
\eitm
\end{proof}

\begin{clm}\label{sigmaexists}
Equation \eqref{imro2d:sig:quad} has a real solution, i.e., $\sigma^k$ exists.
\end{clm}
\begin{proof}
We want to show that $\eta_2^2-4\eta_1\eta_3\ge0$. Note that\\
\beqa \label{imro2d:exsig:prf1:e1}
\eta_2^2-4\eta_1\eta_3&=& \l(S_{11}^2 + S_{22}^2 + {4(\eps^k)^2 S_{12}^2} + {2S_{11}S_{22}} -4\eps^k S_{11}S_{12} -4\eps^k S_{22}S_{12}\r) \nonumber \\
&+&\l({{-4S_{11}S_{22}}}+4S_{12}^2+4(\eps^k)^2 S_{11}S_{22} {-4(\eps^k)^2 S_{12}^2}\r).
\eeqa
By making the following substitutions in \eqref{imro2d:exsig:prf1:e1}
\begin{align}
 S_{11}^2&=(1-(\eps^k)^2)S_{11}^2+(\eps^k)^2 S_{11}^2, \\
 S_{22}^2&=(1-(\eps^k)^2)S_{22}^2+(\eps^k)^2 S_{22}^2,\\
 4(\eps^k)^2 S_{11}S_{22}&= 2(\eps^k)^2 S_{11}S_{22}+2(\eps^k)^2 S_{11}S_{22},
  \end{align}
we will get
\beqan 
\eta_2^2-4\eta_1\eta_3&=&(1-(\eps^k)^2)S_{11}^2+{(\eps^k)^2 S_{11}^2}+(1-(\eps^k)^2)S_{22}^2+{(\eps^k)^2 S_{22}^2} -{4(\eps^k) S_{11}S_{12}}\notag \\
& &- {4(\eps^k) S_{22}S_{12}}-2S_{11}S_{22} +{4S_{12}^2} +{2(\eps^k)^2 S_{11}S_{22}}+2(\eps^k)^2 S_{11}S_{22}\notag \\
&=&(\eps^k S_{11} +\eps^k S_{22} -2 S_{12})^2 + \l(1-(\eps^k)^2\r) ( S_{11}-S_{22})^2 \ge 0. \notag
\eeqan
\end{proof}

%%%
\begin{clm}\label{imro2d:uex:clm1}
$\sigma^k \ge S_{11}$ and $\sigma^k \ge S_{22}$; therefore $u^k$ exists.
\end{clm}
\begin{proof}
We will prove it for $S_{11}$; the proof for $S_{22}$ would be similar. \\
\beqan
(\eps^k S_{11}-S_{12})^2 \ge 0 & \Rightarrow& S_{12}^2 \ge -(\eps^k)^2 S_{11}^2 + 2\eps^k S_{11}S_{12}\notag, \\
& \Rightarrow& S_{12}^2 - S_{11}S_{22}
\ge -(\eps^k)^2 S_{11}^2 + 2\eps^k S_{11}S_{12}- S_{11}S_{22}+S_{11}^2-S_{11}^2\notag, \\
& \Rightarrow& -\eta_3 \ge S_{11}^2\eta_1 + S_{11}\eta_2\notag, \\
& \Rightarrow& -4\eta_1\eta_3 \ge 4\eta_1\l(S_{11}^2\eta_1+ S_{11}\eta_2\r)= 4S_{11}^2\eta_1^2 + 4S_{11}\eta_1\eta_2\notag, \\
& \Rightarrow& \eta_2^2-4\eta_1\eta_3 \ge 4S_{11}^2\eta_1^2 + 4S_{11}\eta_1\eta_2+\eta_2^2 = \l(2S_{11}\eta_1+\eta_2\r)^2,\notag \\
&\Rightarrow& \sqrt{ \eta_2^2-4\eta_1\eta_3} \ge 2S_{11}\eta_1+\eta_2,\notag \\
&\Rightarrow&\sigma^k = \frac{-\eta_2+\sqrt{ \eta_2^2-4\eta_1\eta_3}}{2\eta_1} \ge S_{11}.\notag 
\eeqan
\end{proof}

\begin{clm}Suppose $\sigma^k$ and $u^k$ are as defined in IMRO-2D by \eqref{imro2d:sig:quad} %\eqref{imro2d:getu:eq1}  
and \eqref{imro2d:getu:linsys}. Then $H^k=\l(\sigma^k\mathbf I - u^k(u^k)^t\r) \succeq 0$.
\end{clm}
\begin{proof}
We will prove that $\sigma^k \ge \|u^k\|^2$, which implies that $\sigma^k\|x\|^2 \ge \|u^k\|^2\|x\|^2 \ge \l((u^k)^tx\r)^2$ for all $x \in \mathbb R^n$; thus $H^k\succeq 0$.

Recall that $u^k=\tau^k \nabla  \tilde f^k +\rho^k d^k$, $\| \nabla \tilde f^k\|=\|d^k\|=1$ and $\eps^k=({\nabla \tilde f^k})^td^k$ by definition, so 
\beq \|u^k\|^2 = (\tau^k)^2+(\rho^k)^2 + 2\tau^k \rho^k \eps^k. \eeq
In addition, recall \eqref{imro2d:getu:eq1} in which we had 
\beqa \notag
\sigma^k - S_{11}&=& (\tau^k)^2 + 2\epsilon^k \tau^k \rho^k  + (\epsilon^k)^2 (\rho^k)^2, 
\\
\epsilon^k \sigma^k - S_{12}&=& \epsilon^k (\tau^k)^2 + \tau^k \rho^k+ (\epsilon^k)^2 \tau^k \rho^k  + (\epsilon^k) (\rho^k)^2, \notag
\\
 \sigma^k - S_{22}&=& (\epsilon^k)^2 (\tau^k)^2 + 2\epsilon^k \tau^k \rho^k  +  (\rho^k)^2 \notag.
\eeqa
Let us multiply the second equation by $-2\eps^k$ and add the result to the summation of the other two equations to get 
\beqan
\text{lhs}&=& \sigma^k -S_{11} -2(\eps^k)^2 \sigma^k + 2\eps^k S_{12} + (\sigma^k) -S_{22}= 2\l(1-(\eps^k)^2\r)\sigma^k + \eta_2 = 2 \eta_1 \sigma^k +\eta_2 \notag, \\
\text{rhs}&=& {(\tau^k)^2} + {{2 \eps^k\tau^k \rho^k}} + {(\eps^k)^2 (\rho^k)^2} - {2(\eps^k)^2 (\tau^k)^2} -{{2\eps^k \tau^k \rho^k}} -{2(\eps^k)^3 \tau^k \rho^k} \\
&&- {2(\eps^k)^2(\rho^k)^2 }+ {(\eps^k)^2 (\tau^k)^2} + {2(\eps^k) \tau^k \rho^k }+ {(\rho^k)^2}\notag \\
&=&(\tau^k)^2(1-(\eps^k)^2)+2\eps^k \tau^k \rho^k (1-(\eps^k)^2) + (\rho^k)^2 \l(1-(\eps^k)^2\r)\\
&=& \l(1-(\eps^k)^2\r)\|u^k\|^2= \eta_1\|u^k\|^2\notag, \\
&\Rightarrow& 2\eta_1\sigma^k +\eta_2= \eta_1\|u^k\|^2 \notag, \\
&\Rightarrow& {-\eta_2}+\sqrt{\eta_2^2-4\eta_1\eta_3}+{\eta_2} = \eta_1\|u^k\|^2 \notag, \\
&\Rightarrow& \|u^k\|^2=\frac{\sqrt{\eta_2^2-4\eta_1\eta_3}}{\eta_1}.\notag 
\eeqan
By the value of $\sigma^k$, we have 
\beqn
\sigma^k -\|u^k\|^2 = \frac{-\eta_2+\sqrt{\eta_2^2-4\eta_1\eta_3}}{2\eta_1}-\frac{\sqrt{\eta_2^2-4\eta_1\eta_3}}{\eta_1}=\frac{-\eta_2-\sqrt{\eta_2^2-4\eta_1\eta_3}}{2\eta_1}\ge 0,
\eeqn
where the final inequality holds by property \eqref{imro2d:prop1}, i.e., 
\beqn
\eta_1\ge 0, \ \eta_3\ge 0 \Rightarrow -4\eta_1\eta_3\le 0,
\eeqn
 hence
 \beqn
  \eta_2^2-4\eta_1\eta_3\le \eta_2^2 \Rightarrow \sqrt{\eta_2^2-4\eta_1\eta_3}\le |\eta_2|=-\eta_2.
 \eeqn
\end{proof}

Note that $\sigma^k > \|u^k\|^2$ unless $\eta_1=0$ (i.e. $\eps^k=0$) or $\eta_3=0$ (i.e. $\det(S)=0$). Both cases happen only if $\nabla \tilde f^k$ and $d^k$ are parallel, otherwise $H^k= \sigma^k \mathbf I -u^k(u^k)^t \succ 0$.

Before we start the analysis on the convergence of IMRO, we would like to point out that IMRO-2D reduces to linear CG (LCG) in the absence of the regularizer's term, i.e., $\lambda \|x\|_1$. 
The following theorem explains why IMRO-2D is essentially linear CG when the regularization term is missing. Note that the theorem is not trivial because, although CG and IMRO-2D form a model in the same two dimensional subspace, IMRO-2D looks for the minimizer in the full space. In the $\lambda=0$ case, the theorem shows that the minimizer occurs in the same subspace.
\begin{thm}
Suppose IMRO-2D is applied to minimizing the quadratic function $\frac{1}{2} \|Ax-b\|^2$. Then the sequence of iterates generated by IMRO-2D is the same as iterates generated in linear CG.
\end{thm}
\begin{proof}
Notice that 
$$\tilde f(x)=\frac{1}{2} \|Ax-b\|^2= \frac{1}{2}x^tA^tAx - x^tA^tb + \frac{1}{2} b^tb.$$
Let $Q$ and $c$ denote $A^tA$ and $A^tb$, respectively. The proof is by induction. Let $x^0$ and $r^0=\nabla \tilde f(x^0)=Qx^0-c$ be the starting point for both algorithms. The subscript CG distinguishes iterates for LCG from iterates obtained by IMRO-2D.  
For IMRO-2D at the first iteration we have 
$$\frac{1}{2}({r^0})^t (\sigma^0 \mathbf I) r^0 = \frac{1}{2} ({r^0})^t A^tAr^0\ \  \Rightarrow \ \ \sigma^0= \frac{(r^0)^t A^tAr^0}{(r^0)^tr^0},$$
and $$x^1= x^0 + \frac{1}{\sigma^0}(-r^0).$$
By the fact that for LCG, $p_{CG}^0=-r^0$ and $\alpha^0=\frac{(r^0)^tr^0}{(r^0)^tA^tAr^0}=\frac{1}{\sigma^0}$, we get $x_{CG}^1= x^1$. Suppose this holds true for $k$, i.e., $x^k_{CG}=x^k$. Because $m_{H^k}(x,x^k)= \tilde f(x)$ on the space of $x^k + \spn\l\{ r^k, d^k\r\}$, if $x^{k+1} \in x^k + \spn\l\{ r^k, d^k\r\}$ then $x^{k+1}$ must be $x_{CG}^{k+1}$. Therefore it suffices to show that $x^{k+1} \in x^k + \spn\l\{ r^k, d^k\r\}$.

Using optimality condition for our model $m_{H^k}(x,x^k)$ we get that 
\beqan
x^{k+1}&=&x^k - (H^k)^{-1}r^k= x^k - \l( \sigma^k \mathbf I - u^k(u^k)^t\r)^{-1} r^k\\
&=& x^k - \l( \frac{1}{\sigma^k}\mathbf I - 
\frac{1}{\sigma^k(\|u^k\|^2-\sigma^k)} u^k(u^k)^t \r) r^k\\
&=& x^k - \frac{1}{\sigma^k}r^k + 
\frac{(u^k)^t  r^k}{\sigma^k(\|u^k\|^2-\sigma^k)} (\tau^k r^k + \rho^k d^k) \in \l\{x^k + \spn\{r^k,d^k\}\r\}.
\eeqan
\end{proof}

\newpage
The general framework of IMRO is captured in the following algorithm.

\begin{alg}\label{imroalg}
\begin{tabbing}
++\= +++\= +++\= +++\= +++ \= \kill 
\>Let $x^0 \in \mathbb R^n$ be an arbitrary starting point and $x^1 =\prox_{\tilde p}(x^0 - \nabla \tilde f^0/\sigma^0)$. \\
\>\> {\bf for} $k=1,2,\ldots$\\
\>\>\> Find $\sigma^k$ and $u^k$:\\
\>\>\>\> equations \eqref{imro1dsig} and \eqref{imro1du} for IMRO-1D \\
\>\>\>\> equations \eqref{imro2d:sig:quad} and \eqref{imro2d:getu:linsys} for IMRO-2D \\
\>\>\> Find $x^{k+1}$ using Algorithm \ref{algfindx}\\
\>\>\> Update $\nabla \tilde f^k$ and $d^k$
\end{tabbing}
\end{alg}

Note that we have not explained the choice of vector $v^k$ (which determines $u^k$) in IMRO-1D. Our choice of $v^k$ is discussed in Section \ref{ch:imroNR}.
The termination criterion used for IMRO is the measurement on the norm of subgradient of function $\tilde F(x)$, $\xi$, which at iteration $k$ is 
\beqan
\xi_i &=\lambda + \nabla \tilde f^k_i  \quad & \text {if} \ x_i>0,\\
\xi_i &=-\lambda + \nabla \tilde f^k_i  \quad & \text {if} \ x_i<0,\\
\xi_i &=-\lambda \alpha + \nabla \tilde f^k_i \quad & \text{if} \ x_i=0 \ \ \text{for some} \ \ \alpha\in [-1,1].
\eeqan

Note that for zero entries, if $|\nabla \tilde f^k_i| \le \lambda $, then there exists an $\alpha$ such that $\xi_i=0$. In case that $|\nabla \tilde f^k_i| > \lambda$, then $\xi_i \neq 0$; so we take $\alpha$ such that $\xi_i$ is minimized, i.e., $\xi_i=|\nabla \tilde f_i^k| - \lambda$. Therefore,
the norm of subgradient at $x^k$ is easily calculated with no extra computational cost. 

%%%%
\section{Convergence of IMRO}\label{imro:cnvg}
The difference between IMRO and other proximal quasi-Newton methods is the special structure of $H^k$. The format of $H^k$ in IMRO facilitates computation of the next iterate as mentioned earlier. The convergence properties of IMRO, however, can mostly be generalized to other variants of proximal quasi-Newton methods. 

In the preceding sections, we established that $H^k\succeq 0$ for both IMRO-1D and IMRO-2D. Furthermore, the conditions under which $H^k$ is singular are apparently unusual (that $v^k$ is a dominant singular vector of $A$ in the case of IMRO-1D; that $\nabla \tilde f^k$ and $d^k$ are
parallel in the case of IMRO-2D) and never arose in our computational experiments. Therefore, for the remainder of this section, we assume $H^k\succ 0$. If one of these unusual cases arose in practice, we could simply modify the algorithm by replacing  $\sigma^k$ with  $\sigma^k+\epsilon$ for some small $\epsilon>0$ to ensure that $H^k\succ 0$. Consider problem \eqref{mincompfunintro}, where $p(x)$ is a possibly nonsmooth convex function, and $m_{H^k}(x,x^k)$ as defined in \eqref{quadmodelH}. Our notation is as follows:
\beqa \label{imro:cnvg: model}
M_{H^k}(x,x^k):= m_{H^k}(x,x^k) + p(x) ,
\eeqa
and 
\beqa\label{imrocnvgscheme}
M_{H^k}^\ast &:=& \min_{x} \ \  M_{H^k}(x,x^k),\notag \\
x_{H^k}^\ast &:=& \arg\min_{x}\ \  M_{H^k}(x,x^k),\\
g_{H^k}  & := & H^k( x^k - x_{H^k}^\ast). \notag
\eeqa
Note that optimality conditions for \eqref{imro:cnvg: model} implies that 
\beq \label{imro:cnvg:geq1}
g_{H^k}= \nabla f^k + \xi^{k+}, 
\eeq
where $\xi^{k+} \in \partial\l(p(  x^\ast_{H^k})\r)$. We will see in this section that the notion of scaled gradient, $g_{H^k}$, mimics some of the properties of gradient. An important property of $g_{H^k}$ is captured below. 
\begin{prop}
$g_{H^k}=0$ if and only if $x^k$ is the optimizer of problem \eqref{mincompfunintro}. 
\end{prop}
\begin{proof}
Note that if $g_{H^k}=0$, then $x^k -  x^\ast_{H^k} = 0$ because $H^k\succ 0$ (thus invertible). This implies that $x^k=  x^\ast_{H^k}$. Therefore \eqref{imro:cnvg:geq1} reduces to optimality condition for \eqref{mincompfunintro}. Likewise if $x^k$ is the optimal solution of \eqref{mincompfunintro}, then $\nabla f^k + \xi^k=0$ implies that $x^\ast_{H^k} =x^k$; thus $g_{H^k}=0$.
\end{proof}

The following lemma (which is partially based on \cite[Proposition 2.3]{ProxNMcompfun}) shows that in fact direction $x^\ast_{H^k} - x^k$ is a descent direction; in other words using this direction armed with a line search we attain the next iterate, $x^{k+1}$, for which we have $F(x^{k+1}) < F(x^k)$. 

\begin{lem} \label{lem:decsq} Suppose $x^\ast_{H^k}$ is as defined in \eqref{imrocnvgscheme} for some $H^k\succ 0$. Let the iterative scheme of
\beq\label{ItrSchm}
x^{k+1}= x^k + \alpha^k (x^\ast_{H^k} - x^k)
\eeq
be applied to problem \eqref{mincompfunintro} while $x^k$ is not the optimizer of \eqref{mincompfunintro}.
Then the following properties hold true:
\begin{enumerate}
\item $F(x^{k+1}) < F(x^k)$ for sufficiently small step size $\alpha^k>0$.
\item In addition, if \beq \label{imro:flessthanm}
F(x^\ast_{H^k}) \le M_{H^k}(x_{H^k}^\ast, x^k).  
\eeq 
holds for $x^\ast_{H^k}$, then $F(x^{k+1}) < F(x^k)$, for any $\alpha^k\in (0,1]$.
\end{enumerate}
\end{lem}
\begin{proof}
Let us denote $x^\ast_{H^k} - x^k$ by $d^k$. Since $H^k \succ 0$, $ x^\ast_{H^k}$ is the unique optimizer of \eqref{imrocnvgscheme} and $x^\ast_{H^k}\ne x^k$ by assumption. Hence we have 
\beqan
\frac{1}{2} \|d^k\|_{H^k}^2 & + & \langle \nabla f^k, d^k\rangle + f(x^k) + p( x^\ast_{H^k})\\
& <& \frac{(\alpha^k)^2}{2} \|d^k\|_{H^k}^2 + \alpha^k\langle \nabla f^k, d^k\rangle + f(x^k) + p(x^{k+1})\notag \\
&\le& \frac{(\alpha^k)^2}{2} \|d^k\|_{H^k}^2 + \alpha^k\langle \nabla f^k, d^k\rangle + f(x^k) + \alpha^k p( x^\ast_{H^k})+ (1-\alpha^k)p(x^{k}),\notag 
\eeqan
where the last inequality follows from convexity of $p(x)$. Rearranging the terms and dividing by $1-\alpha^k$ we get 
\beq \label{imro:cnvg:redprfeq1}
\langle \nabla f^k, d^k\rangle + p(x^\ast_{H^k}) - p(x^k) < - \frac{1+\alpha^k}{2} \|d^k\|_{H^k}^2.
\eeq
Using the convexity of $p(x)$ and the Taylor expansion for $f(x)$, we derive 
\beqa
F(x^{k+1}) - F(x^k) &=& f(x^{k+1}) - f(x^k) + p(x^{k+1}) - p(x^k) \nonumber \\
&\le & \alpha^k \langle \nabla f^k, d^k\rangle + O\l(({\alpha^k})^2\r) + \alpha^k p(x^\ast_{H^k}) + (1-\alpha^k) p(x^k) - p(x^k)\nonumber \\
&=& \alpha^k \l[ \langle \nabla f^k, d^k\rangle + p( x^\ast_{H^k}) -p(x^k)\r] + O\l((\alpha^k)^2\r) < 0,\label{imro:cnvg:redprfeq2}
\eeqa
by \eqref{imro:cnvg:redprfeq1} for sufficiently small values of $\alpha^k$. This concludes the proof of the first property. \\
By our assumption and the fact that  $x_{H^k}^\ast$ is the unique minimizer of $M_{H^k}(x,x^k)$, we get 
$$F(x^\ast_{H^k}) \le M_{H^k}(x_{H^k}^\ast,x^k) < M_{H^k}(x^k,x^k) = F(x^k).$$
We now attain the desired result using the convexity of $F(x)$:
$$F(x^{k+1}) \le \alpha^k F(x^\ast_{H^k}) + (1-\alpha^k) F(x^k)< F(x^k).$$
\end{proof}

The first property in Lemma \ref{lem:decsq} establishes that both variants of IMRO generate a decreasing sequence when paired with a line search. Moreover, the condition in the second property always holds true for IMRO-1D meaning that IMRO-1D generates a decreasing sequence for any step size regardless of the employed line search.
\\

The following lemma which is similar to \cite[Lemma 2.3]{BT:fista} is the essence of showing the convergence properties of IMRO-1D. Here, we assume that no line search is employed, i.e., the step size of $\alpha^k$ in \eqref{ItrSchm} is 1. 
\begin{lem}\label{imro:cnvg:LBonFlem1}
Let $x^\ast_{H^k}$ and $g_{H^k}$ be as defined in \eqref{imrocnvgscheme} for some $H^k\succ 0$.
Suppose the scheme of 
\beq\label{ItrSchm2}
x^{k+1}=x^\ast_{H^k}
\eeq
is applied to problem \eqref{mincompfunintro}, and 
$$F(x^{k+1})\le M_{H^k}(x^{k+1},x^k).$$ 
Then the following inequalities hold:
\begin{enumerate}
\item
\beq \label{imro:cnvg:lbeq1} F(x)-F(x^{k+1})\ge \frac{1}{2} \|x^{k+1}-x^k\|_{H^k}^2 + \langle g_{H^k}, x-x^k \rangle, \ \text{for } \forall x \in \mathbb R^n.\eeq
\item \beq \label{imro:cnvg:lbeq2}
F(x^k)-F(x^{k+1}) \ge \frac{\l(F(x^{k+1})-F(\xs)\r)^2}{2(\bar \lambda)^2 \delta^2},
\eeq
where $\delta$ is the diameter of the level set of $x^0$, $\bar \lambda= \max_{i=1,\ldots,k}\lambda_{\max} (H^i)$, and $x^\ast$ is the optimal solution of \eqref{mincompfunintro}.

\end{enumerate}
\end{lem}
\begin{proof}
We start with showing the first inequality, then using this inequality we will prove the second one. 
Recall that we have
\beq \label{imro:cnvg:prflem1:eq1}
H^k (x^{k+1}-x^k) +\nabla f(x^k) + \xi^{k+1} =0,
\eeq
where $\xi^{k+1} \in \partial \l(p(x^{k+1})\r)$.
By hypothesis we have 
\beq\label{imro:cnvg:prflem1:eq2}
F(x)-F(x^{k+1}) \ge F(x)-M_{H^k}(x^{k+1},x^k),
\eeq 
and by convexity of $f(x)$ and $p(x)$ we derive
\beqa 
f(x)&\ge & f(x^k)+ \langle \nabla f(x^k), x-x^k\rangle, \label{imro:cnvg:prflem1:eq2b}\\ 
p(x)&\ge & p(x^{k+1}) + \langle \xi^{k+1} , x-x^{k+1}\rangle.
\eeqa
Summing the above inequalities, we get
\beq\label{imro:cnvg:prflem1:eq3}
F(x) \ge f(x^k) + \langle \nabla f(x^k), x-x^k\rangle +p(x^{k+1}) + \langle  \xi^{k+1} , x-x^{k+1}\rangle.
\eeq
Substituting \eqref{imro:cnvg:prflem1:eq3} in \eqref{imro:cnvg:prflem1:eq2} gives us
\beqan
F(x)-F(x^{k+1})&\ge & {f(x^k)} + \langle \nabla f(x^k), x-x^k\rangle +{p(x^{k+1})} + \langle \xi^{k+1} , x-x^{k+1}\rangle \notag \\
& &- \frac{1}{2} \|x^{k+1} -x^k\|_{H^k}^2 - \langle \nabla f(x^k), x^{k+1}-x^k\rangle -{ f(x^k)} -{p(x^{k+1})}\notag \\
&=&  -\frac{1}{2}\|x^{k+1} -x^k\|_{H^k}^2+ \langle \nabla f(x^k) +  \xi^{k+1} , x-x^{k+1}\rangle\notag \\
&=&-\frac{1}{2}\|x^{k+1} -x^k\|_{H^k}^2 +\langle H^k(x^{k} -x^{k+1}), x- x^k + (x^k-x^{k+1})\rangle\notag \\
&=&\frac{1}{2}\|x^{k+1} -x^k\|_{H^k}^2 +\langle g_{H^k}, x- x^k \rangle.\notag 
\eeqan
This concludes our proof for the first inequality. Now, let us apply inequality \eqref{imro:cnvg:lbeq1} at $x=\xs$ to get
\beqa
F(\xs)-F(x^{k+1}) &\ge& \frac{1}{2} \|x^{k+1}-x^k\|^2_{H^k} + \langle H^k(x^k-x^{k+1}) , \xs - x^k\rangle \notag \\
&=& \frac{1}{2} \l( \langle H^k(x^k - x^{k+1}), x^k- x^{k+1}+2\xs -2x^k\rangle\r) \notag \\
&=& \frac{1}{2} \l( \langle H^k(\xs- x^{k+1})-H^k(\xs-x^k), (\xs-x^{k+1})+(\xs -x^k)\rangle\r) \notag \\
&=&\frac{1}{2}\l(\|\xs-x^{k+1}\|^2_{H^k}-\|\xs-x^{k}\|^2_{H^k}\r) \notag \\
&=& \frac{1}{2} \l(\|\xs-x^{k+1}\|_{H^k}-\|\xs-x^{k}\|_{H^k}\r) \l(\|\xs-x^{k+1}\|_{H^k}+\|\xs-x^{k}\|_{H^k}\r)\notag \\
&\ge &\frac{1}{2} \l(-\|x^{k+1}-x^k\|_{H^k}\r) \l(\|\xs-x^{k+1}\|_{H^k}+\|\xs-x^{k}\|_{H^k}\r)\notag,
\eeqa
where the last line is by triangle inequality. Therefore, we have
\beq \label{prfineq1}
\|x^{k+1}-x^k\|_{H^k}\ge \frac{-2\l(F(\xs)-F(x^{k+1})\r)}{\|\xs-x^{k+1}\|_{H^k}+\|\xs-x^{k}\|_{H^k}}\ge \frac{F(x^{k+1})-F(\xs)}{\bar \lambda \delta},
\eeq
by the fact that $\|x^k-\xs\|,\|x^{k+1}-\xs\| \le \delta$ because $F$ is convex and by Lemma \ref{lem:decsq} the objective values of the iterates in \eqref{ItrSchm2} is descending, and $\|x\|_{H^k}\le \bar \lambda \|x\|$ by definition of $\bar \lambda$.

Moreover, by inequality \eqref{imro:cnvg:lbeq1} at $x^k$ we have 
\beq
F(x^k) -F(x^{k+1}) \ge \frac{1}{2}\|x^{k+1}-x^k\|_{H^k}^2.
\eeq

Applying inequality \eqref{prfineq1} concludes the result we wanted to show.\\
\end{proof}

Recall that by definition of $g_{H^k}$, and the fact that $H^k \succ 0$, we get 
\beq \label{HinvSclGrd}
{x^\ast_{H^k}}-x^k= (H^k)^{-1} g_{H^k}.
\eeq 
Furthermore, by $H^k\preceq \lambda_{\max}(H^k)\mathbf I$, where $\lambda_{\max}$ is for maximum eigenvalue, we get $(H^k)^{-1}\succeq \frac{1}{\lambda_{\max}(H^k)} \mathbf I$; hence
\beq \label{ninvHlb}
\|x\|_{(H^k)^{-1}}^2\ge \frac{1}{\lambda_{\max}(H^k)} \|x\|^2,
\eeq
and we conclude the following corollary.  
\begin{cor}\label{imro:cnvg:cor0}
Let $x^\ast_{H^k}$ be as defined in \eqref{imrocnvgscheme} for $H^k \succ 0$. Suppose $x^{k+1}= x^\ast_{H^k}$, and $F(x^{k+1})\le M_{H^k}(x^{k+1},x^k)$. Then for $\forall x\in \mathbb R^n$ we have
\beqa
F(x)&\ge & F(x^{k+1})+ \langle g_{H^k}, x-x^k \rangle +\frac{1}{2} \|g_{H^k}\|_{(H^k)^{-1}}^2 \\
&\ge & F(x^{k+1})+ \langle g_{H^k}, x-x^k \rangle +\frac{1}{2\lambda_{\max}(H^k)} \|g_{H^k}\|^2. \label{suffRed4PQN}
\eeqa
\end{cor}
\begin{proof}
Immediately follows from inequalities \eqref{imro:cnvg:lbeq1} in Lemma  \ref{imro:cnvg:LBonFlem1}, \eqref{HinvSclGrd}, and \eqref{ninvHlb}.
\end{proof}

Note that in IMRO-1D, $H^k\preceq \sigma^k \mathbf I =L \mathbf I$ and $(H^k)^{-1}\succeq \frac{1}{\sigma^k} \mathbf I= \frac{1}{L} \mathbf I$. As a result 
inequality \eqref{imro:cnvg:lbeq2} reduces to
\beq \label{imro:cnvg:lbeq2imro1d}
F(x^k)-F(x^{k+1}) \ge \frac{\l(F(x^{k+1})-F(\xs)\r)^2}{2 L^2 \delta^2};
\eeq
and we get the following inequality 
by applying Corollary \ref{imro:cnvg:cor0} at $x=x^k$,
\beq \label{suffRed4SclGrd}
F(x^{k+1})\le F(x^k)- \frac{1}{2L} \|g_{H^k}\|^2.
\eeq

Inequality \eqref{suffRed4SclGrd} clarifies more similarities between scaled gradient, $g_{H^k}$, and the notion of gradient in smooth unconstrained problems. One of the helpful properties of an algorithm for unconstrained smooth optimization is to have a sufficient reduction in the objective value at each iteration, i.e., 
\beq
f(x^{k+1})\le f(x^k)- \frac{1}{2L} \|\nabla f^k\|^2.
\eeq
Equivalently, inequality \eqref{suffRed4SclGrd} implies 
the sufficient reduction in objective value at each iteration for IMRO-1D. In general, Corollary \ref{imro:cnvg:cor0} captures the sufficient reduction condition for a proximal quasi-Newton method provided that $F(x^{k+1})\le M_{H^k}(x^{k+1},x^k)$ holds. 
%%%%
The sublinear convergence of IMRO-1D is established in the following lemma.\\

\begin{lem}\label{imro:cnvg:mainlemma}
Suppose $\{\omega^k\}\rightarrow \omega^\ast$ is a decreasing sequence; $\omega^k -\omega^{k+1} \ge \frac{(\omega^{k+1} - \omega^\ast)^2}{\mu}$ for all $k$; and $\omega^1-\omega^\ast \le 4\mu$. Then for all $k$ we have 
\beq \label{wSqeq1}
\omega^k-\omega^\ast \le \frac{4\mu}{k}.
\eeq
\end{lem}
\begin{proof}
Proof is by induction. For $k=1$ the result holds by the hypothesis. Suppose \eqref{wSqeq1} holds for $k$; and let $p_k$ denote $p(k):= \frac{4\mu}{k}$, then 
\beqa
\omega^{k+1}-\omega^\ast &=&\omega^{k}-\omega^\ast + \omega^{k+1}-\omega^k  \notag \\
&\le& \omega^{k}-\omega^\ast - \frac{(\omega^{k+1} - \omega^\ast)^2}{\mu}  \notag \\
&\le& p_k - \frac{(\omega^{k+1} - \omega^\ast)^2}{\mu}.  \notag 
\eeqa
Let $\nu= \omega^{k+1} -\omega^\ast$. Then the above inequality is 
\beq
\frac{\nu^2}{\mu} + \nu - p_k \le 0, 
\eeq
which has a nonnegative solution given by 
\beq
\nu\le \frac{-\mu+\sqrt{\mu^2+4p_k\mu}}{2} = \frac{2p_k}{1+\sqrt{1+ \frac{4p_k}{\mu}}}.
\eeq
Note that function $f(x)=\frac{1}{1+\sqrt{1+x}}$ is convex; thus on any interval $[0,\alpha]$, it is bounded above by its secant interpolant. We now consider two separate cases; when $k = 1$ to show that lemma holds for $k + 1 = 2$,
and when $k\ge 2$ to show that the lemma holds for $k+1\ge 3$. For $k=1$ we have $p_1\le 4\mu$, so $\frac{4p_1}{\mu}\le 16$, therefore
\beq
\nu\le 2p_1 \l( \frac{1}{2} + \frac{\frac{1}{1+\sqrt{17}}- \frac{1}{2}}{16}\frac{4p_1}{\mu}\r)\le
\frac{4\mu}{2}.
\eeq
For $k\ge 2$, $p_k\le \frac{4\mu}{k} \le 2\mu$, so $\frac{4p_k}{\mu}\le 8$, hence
\beqa
\nu&\le& 2p_k \l( \frac{1}{2} + \frac{\frac{1}{1+\sqrt{9}}- \frac{1}{2}}{8}\frac{4p_k}{\mu}\r) \notag \\
&=&4\mu \l( \frac{1}{k} + 4 \l(\frac{1}{4}- \frac{1}{2}\r) \frac{1}{k^2} \r)\le \frac{4\mu}{k+1},
\eeqa
where the last inequality follows from the fact that $\frac{1}{k} - \frac{1}{k^2} \le \frac{1}{k+1}$.
\end{proof}\\

\begin{thm}\label{cnvgThrm}
Let $x^\ast_{H^k}$ be as defined in \eqref{imrocnvgscheme}. Suppose iterative scheme of 
$$x^{k+1}= x^\ast_{H^k}$$
is applied to problem \eqref{mincompfunintro}, and 
$$F(x^{k+1})\le M_{H^k}(x^{k+1},x^k).$$ Let $\delta$ be the diameter of the level set of $x^0$ and $\bar \lambda= \max_{i=1,\ldots,k}\lambda_{\max} (H^i)$  as in Lemma \ref{imro:cnvg:LBonFlem1}. Then 
$$ F(x^k)-  F(x^\ast) \le \frac{4\mu}{k}, $$
where $\mu=\max \l(2\bar \lambda^2\delta^2, \frac{ F(x^1)- F(x^\ast)}{4} \r)$.
\end{thm}

\begin{proof}
By Lemma \ref{imro:cnvg:LBonFlem1} we have 
$$
F(x^k)- F(x^{k+1}) \ge \frac{\l( F(x^{k+1})- F(\xs)\r)^2}{2(\bar \lambda)^2 \delta^2}
\ge \frac{\l( F(x^{k+1})- F(\xs)\r)^2}{\mu}. 
$$ 
Letting $\omega^k = F(x^k)$, the result of this theorem is a direct implication of Lemma \ref{imro:cnvg:mainlemma}.
\end{proof}\\
\begin{cor}
Let $x^\ast_{H^k}$ be as defined in \eqref{imrocnvgscheme}; and the iterative scheme of 
$$x^{k+1}= x^\ast_{H^k}$$
is applied to problem \eqref{imro:BPDNdef} for $\{H^k\}$ generated by IMRO-1D. Then 
$$\tilde F(x^k)- \tilde F(x^\ast) \le \frac{4\mu}{k}, $$
where $\mu=\max \l(2L^2\delta^2, \frac{\tilde F(x^1)- \tilde F(x^\ast)}{4} \r)$.
\end{cor}
\begin{proof}
Recall that in IMRO-1D, $ A^tA \preceq H^k$, thus $\tilde F(x^{k+1}) \le M_{H^k}(x^{k+1},x^k)$; moreover, $\lambda_{\max} (H^k) \le \sigma^k = L$. Our result is, therefore, concluded by Theorem  \ref{cnvgThrm}.
\end{proof}

%%%%%%%%%%%%
\section{FIMRO - Accelerated Variant of IMRO} \label{sec:fimro}
In this section, we discuss how we may apply the accelerated technique of Nesterov \cite[Chapter 2]{nesterov:intro2CO} to IMRO-1D.   We assume in this section that
$\sigma^k\ge L=\Vert A\Vert^2$.

As in Nesterov's method, we have two sequences $\{y^k\}$ and $\{x^k\}$ in this section. The model $M_{H^k}(x,y^k)$ is built using $y^k$, while its solution generates $\{x^k\}$. In other words we have the following:
\beqa
M_{H^k}^\ast(y^k) &:=& \min \ \  M_{H^k}(x,y^k),\\
x_{H^k}^\ast(y^k) &:=& \arg\min\ \  M_{H^k}(x,y^k),\\
g_{H^k}(y^k)  & := & H^k( y^k - {x^\ast_{H^k}(y^k)}) \ \ \text{$\equiv g_{H^k}$ }.
\eeqa

\begin{dfn}\cite[Definition 2.2.1]{nesterov:intro2CO}
A pair of sequences $\{ \phi^k(x)\}$ and $\{\lambda^k\}$,\ $\lambda^k\ge0$ is called an estimate sequence of function $f(x)$ if $\lambda^k \rightarrow 0$ and for any $x\in \mathbb R^n$ and all $k\ge 0$ we  have 
$$\phi^k(x)\le (1-\lambda^k) f(x) + \lambda^k \phi^0(x).$$
\end{dfn}

The following two lemmas which are analogous to \cite[Lemmas 2.2.2 and 2.2.3]{nesterov:intro2CO} summarize how we can construct an estimate sequence.
\begin{lem}
Suppose $\{y^k\}$ is an arbitrary sequence, $\{\alpha^k\}$ is a sequence such that $\alpha^k\in (0,1)$ and $\sum_{k=0}^{\infty}\alpha^k=\infty$, and $\lambda^0=1$. Moreover assume that 
\beqan 
F(x^\ast_{H^k}{(y^k)})\le M_{H^k}(x^\ast_{H^k}{(y^k)}, y^k).
\eeqan 
Then the pair of sequences $\{\lambda^k\}$ and $\{\phi^k(x)\}$ generated as 
\beqan
\lambda^{k+1}&=&(1-\alpha^k)\lambda^k,\\
\phi^{k+1}(x)&=& (1-\alpha^k) \phi^k(x) + \alpha^k \l[ F(x_{H^k}^\ast(y^k))+ \langle g_{H^k}, x-y^k \rangle +\frac{1}{2\sigma^k} \|g_{H^k}\|^2\r]
\eeqan
is an estimate sequence for $F(x)$.
\end{lem}
\begin{proof}
Let $\psi(x) := F(x_{H^k}^\ast(y^k))+ \langle g_{H^k}, x-y^k \rangle +\frac{1}{2\sigma^k} \|g_{H^k}\|^2$. 
Note that by Corollary \ref{imro:cnvg:cor0}, $F(x)\ge \psi(x)$ for $\forall x \in \mathbb R^n$. Our proof is by induction. The base case holds true for $k=0$. Suppose it holds true for $k$, then for $k+1$ we have
\beqan
\phi^{k+1}(x)&=&(1-\alpha^k) \phi^k(x) + \alpha^k \psi(x) \le (1- \alpha^k) \phi^k(x) + \alpha^k F(x)\notag \\
&=& (1-(1-\alpha^k)\lambda^k) F(x) + (1-\alpha^k)\l(\phi^k(x) - (1-\lambda^k) F(x) \r)\notag \\
&\le &(1-(1-\alpha^k)\lambda^k) F(x) + (1-\alpha^k)\lambda^k \phi^0(x)\notag \\
&=& (1-\lambda^{k+1}) F(x) + \lambda^{k+1} \phi^0(x).\notag
\eeqan
\end{proof}

In the following lemma we show how we may write $\phi^{k+1}(x)$ in closed form.

\begin{lem}
Suppose $\phi^0(x)= F(x^0)+\frac{\gamma^0}{2}\|x-z^0\|^2$. Then $\phi^{k+1}(x)$ generated by the recursive formulation of the previous lemma is 
$$\phi^{k+1}(x)= \bar \phi^{k+1}+\frac{\gamma^{k+1}}{2}\|x-z^{k+1}\|^2,$$ 
where 
\beqa
\gamma^{k+1}&=& (1-\alpha^k)\gamma^k, \notag \\
z^{k+1}&=&z^k-\frac{\alpha^k}{\gamma^{k+1}} g_{H^k} \notag, \\
{\bar\phi^{k+1}}&=&(1-\alpha^k){\bar\phi^k}+ \alpha^k F(x_{H^k}^\ast(y^k))+\l( \frac{\alpha^k}{2\sigma^k}-\frac{(\alpha^k)^2}{2\gamma^{k+1}} \r)\|g_{H^k}\|^2 +\alpha^k \langle g_{H^k}, z^k -y^k \rangle.\notag
\eeqa
\end{lem}
\begin{proof}
The proof is by induction. The base case for $k=0$ holds. Suppose for $k$ we have 
$$\phi^k(x)= {\bar\phi^k} + \frac{\gamma^k}{2} \|x-z^k\|^2,$$
then by the previous lemma we get
\beqan
\phi^{k+1}(x)&=&(1-\alpha^k)\l[{\bar \phi^k} + \frac{\gamma^k}{2} \|x-z^k\|^2\r]\\
&&+ \alpha^k\l[F(x_{H^k}^\ast(y^k))+ \langle g_{H^k}, x-y^k \rangle +\frac{1}{2\sigma^k} \|g_{H^k}\|^2 \r].
\eeqan
Using the fact that $\phi$ is a quadratic function we get $$\nabla^2 \phi^{k+1}(x)=\gamma^{k+1}= (1-\alpha^k)\gamma^k .$$
By $\nabla \phi^{k+1}(x) =0 $, we get the minimizer of $\phi^{k+1}$ which is $\frac{1}{\gamma^{k+1}}z^{k+1}$, so 
\beqan
&\nabla \phi^{k+1}(x) = (1-\alpha^k) \gamma^k(x-z^k) + \alpha^k g_{H^k}, \\
\Rightarrow & x= z^{k+1} = \argmin \phi^{k+1}(x) =  z^k-\frac{\alpha^k}{\gamma^{k+1}} g_{H^k} .
\eeqan
To find ${\bar\phi^{k+1}}$ we set equal the $\phi^{k+1}(y^k)$ in both formulations of $\phi^{k+1}$. We, therefore, have
\beqan
{\bar\phi^{k+1}}+ \frac{\gamma^{k+1}}{2} \|y^k &-& z^{k+1}\|^2 ={\bar\phi^{k+1}} + \frac{\gamma^{k+1}}{2} \|z^k-y^k\|^2 -\alpha^k\langle g_{H^k}, z^k-y^k \rangle+\frac{(\alpha^k)^2}{2\gamma^{k+1}}\|g_{H^k}\|^2 \notag \\
&=& (1-\alpha^k){\bar\phi^k} + \frac{(1-\alpha^k)\gamma^k}{2} \|y^k-z^k\|^2+ \alpha^k F(x_{H^k}^\ast(y^k))+ \frac{\alpha^k}{2\sigma^k} \|g_{H^k}\|^2, 
\eeqan
i.e.,
\beq \label{eq0b}
\bar\phi^{k+1}=(1-\alpha^k){\bar\phi^k}+ \alpha^k F(x_{H^k}^\ast(y^k))+\l( \frac{\alpha^k}{2\sigma^k}-\frac{(\alpha^k)^2}{2\gamma^{k+1}} \r)\|g_{H^k}\|^2 + \alpha^k \langle g_{H^k}, z^k -y^k \rangle.
\eeq
\end{proof} 

We would like to construct $y^k$ such that ${\bar\phi^{k+1}} \ge F(x_{H^k}^\ast(y^k))=F(x^{k+1})$. The benefit of this condition will be clear in Theorem \ref{fimro:thm1}. Note that for $k=0$, ${\bar\phi^0} = F(x^0)$ and the condition holds. Let $x^{k+1}=x_{H^k}^\ast(y^k)$, and suppose the required condition is satisfied for $k$, i.e. ${\bar\phi^k} \ge F(x^k)$. Using Corollary \ref{imro:cnvg:cor0} at $x^k=y^k$ and $x=x^k$ we derive
\beq\label{eq0}
{\bar\phi^k}\ge F(x^k)\ge F(x_{H^k}^\ast(y^k)) + \langle g_{H^k}, x^k-y^k\rangle + \frac{1}{2\sigma^k}\|g_{H^k}\|^2.
\eeq
Substituting inequality \eqref{eq0} in the equation \eqref{eq0b} we get
\beqa
{\bar\phi^{k+1}}  \ge  F(x_{H^k}^\ast(y^k))&+&\l( \frac{1}{2\sigma^k}-\frac{(\alpha^k)^2}{2\gamma^{k+1}} \r)\|g_{H^k}\|^2 \notag \\
& +& \langle g_{H^k}, \alpha^k(z^k -y^k)+(1-\alpha^k)(x^k-y^k) \rangle.
\eeqa
To make sure that ${\bar\phi^{k+1}} \ge  F(x_{H^k}^\ast(y^k))$, we need to set 
\beqa
\sigma^k (\alpha^k)^2& =& (1-\alpha^k) \gamma^k \ \equiv \gamma^{k+1}, \label{fimro:gmul0eq1}\\
y^k&=&\alpha^k z^k + (1-\alpha^k) x^k \label{fimro:linterm0eq1}.
\eeqa
Equation \eqref{fimro:gmul0eq1} ensures that the coefficient of $\|g_{H^k}\|^2$ is zero, and \eqref{fimro:linterm0eq1} makes the linear term vanish.  The proposed accelerated scheme is summarized as follows.

\begin{alg}\label{fimro:alg}
\vspace{-0.11 in}
\begin{tabbing}
++\=+++\=+++\=+++\=+++\=\kill\\
\>Let $z^0=x^0$ be arbitrary initial points, $\gamma^0\ge L$ \\
\> {\bf for} i=0,1,\ldots\\
\>\> Compute $\alpha^k$ as $\sigma^k(\alpha^k)^2=(1-\alpha^k)\gamma^k$\\
\>\>\> Let $\gamma^{k+1} = \sigma^k(\alpha^k)^2$\\
\>\> Let $\displaystyle{y^k= \alpha^k z^k + (1-\alpha^k) x^k}$\\
\>\>\> Compute $f(y^k)$ and $\nabla f(y^k)$\\ 
\>\> Find $\sigma^k$ and $u^k$:\\
\>\>\> equations \eqref{imro1dsig} and \eqref{imro1du} (as in IMRO-1D) \\
\>\> Find $x^{k+1}=x^\ast_{H^k}(y^k)$ using algorithm \ref{algfindx}\\
\>\>\>(Note that for this choice of $x$ we get \\
\>\>\>$F(x^{k+1})\le F(y^k)- \frac{1}{2\sigma^k} \|g_{H^k}\|^2\le F(y^k)- \frac{1}{2L} \|g_{H^k}\|^2 $)\\ 
\>\> Let $z^{k+1}=z^k-\frac{\alpha^k}{\gamma^{k+1}} g_{H^k}$
\end{tabbing}
\end{alg}

The following theorem from \cite[Lemma 2.2.1]{nesterov:intro2CO} reveals the importance of condition ${\bar\phi^{k+1}} \ge  F(x_{H^k}^\ast(y^k))$. 

\begin{thm}\label{fimro:thm1}
Suppose $F(x^k)\le{\bar\phi^{k}} $ holds true for a sequence $\{x^k\}$. Then 
$$ F(x^k) - F(\xs) \le \lambda^k \l[ F(x^0)+\frac{\gamma^0}{2}\|x^0-\xs\|^2 - F(\xs)\r].$$
\end{thm}
\begin{proof}
By definition of an estimate sequence we get 
$$F(x^k)\le {\bar\phi^{k}} \le \min_{x} (1-\lambda^k) F(x) + \lambda^k \phi^0(x) \le (1-\lambda^k) F(\xs) + \lambda^k \phi^0(\xs),$$
so
$$ F(x^k) - F(\xs) \le \lambda^k [ \phi^0(\xs) - F(\xs)]=\lambda^k \l[ F(x^0)+\frac{\gamma^0}{2}\|x^0-\xs\|^2 - F(\xs)\r].$$
\end{proof}

The beauty of the above theorem lies in the fact that the convergence of $\{x^k\}$ follows the convergence rate of $\{\lambda^k\}$. It remains to find the convergence rate of $\{\lambda^k\}$. 
\begin{lem}\cite[Lemma 2.2.4]{nesterov:intro2CO} \label{fimro:Bolam}
Suppose $\gamma^0\ge L$. Then 
$$\lambda^k \le \frac{4\sigma^k}{(2\sqrt{\sigma^k}+ k \sqrt{\gamma^0})^2} .
$$
\end{lem}
\begin{proof}
Inductively we show that $\gamma^0 \lambda^k \le \gamma^k$. It holds by our assumption that $\lambda^0=1$ for $0$. Suppose it holds for $k$, then for $k+1$ we get
$$\gamma^0 \lambda^{k+1}= \gamma^0 \lambda^k (1-\alpha^k)\le \gamma^k (1-\alpha^k) = \gamma^{k+1}=\sigma^k (\alpha^k)^2 \ \ \Rightarrow \alpha^k\ge \sqrt{\frac{\gamma^0\lambda^{k+1}}{\sigma^k}}.$$

Let $\tau^k= \frac{1}{\sqrt{\lambda^k}}$, be an increasing sequence; then we have 
\beqan
\tau^{k+1}-\tau^k & =& \frac{\sqrt{\lambda^k}-\sqrt{\lambda^{k+1}}}{\sqrt{\lambda^k \lambda^{k+1}}}=\frac{\lambda^k-\lambda^{k+1}}{\sqrt{\lambda^k \lambda^{k+1}}\l(\sqrt{\lambda^k}+\sqrt{\lambda^{k+1}}\r)}\\
&\ge &\frac{\lambda^k-\lambda^{k+1}}{2\lambda^k\sqrt{\lambda^{k+1}}}=\frac{\lambda^k - (1-\alpha^k)\lambda^k}{2\lambda^k\sqrt{\lambda^{k+1}}}=\frac{\alpha^k}{2\sqrt{\lambda^{k+1}}} \ge \frac{1}{2} \sqrt{\frac{\gamma^0}{\sigma^k}}.
\eeqan
Hence, 
$$\tau^k\ge 1+ \frac{k}{2} \sqrt{\frac{\gamma^0}{\sigma^k}} \ \Rightarrow  \ \lambda\le \frac{4\sigma^k}{(2\sqrt{\sigma^k}+ k \sqrt{\gamma^0})^2}.$$
\end{proof}

The corollary below immediately follows from Theorem \ref{fimro:thm1}, Lemma \ref{fimro:Bolam}, and the fact that $F(x^0)-F(\xs) \le \frac{L}{2}\|x^0-\xs\|^2$.
\begin{cor}Suppose $\gamma^0\ge L$. The generated sequence by Algorithm \ref{fimro:alg} satisfies
$$ F(x^k)-F(\xs)\le \frac{\gamma^0+L}{2} \l ( \frac{4\sigma^k}{(2\sqrt{\sigma^k}+ k \sqrt{\gamma^0})^2} \r)\|x^0-\xs\|^2.$$
Note that in IMRO-1D, $\sigma^k\ge L$. Under the assumption that $\sigma^k=\gamma^0=L$, we have 
$$ F(x^k)-F(\xs)\le   \frac{4L}{(2+ k)^2} \|x^0-\xs\|^2.$$
\end{cor}

This bound of $O(1/k^2)$ residual after $k$ iterations is the best known for
the class of first-order methods for BPDN.

%%%%%%%%%%%%%%
%%%%%%%%%%%%%%
\section{Numerical Result}
\label{ch:imroNR}
We compare IMRO in terms of both speed and accuracy with other available solvers listed below.

\begin{itemize}
\item {\bf GPSR (gradient projection for Sparse Reconstruction)\cite{GPSRp}}
This gradient projection based algorithm first reformulates BPDN problem \eqref{BPDN} into a bound constrained quadratic program (BCQP).
The BCQP formulation is then solved through a projected gradient technique. In other words
$$z^{k+1} = z^k + \lambda^k \l(\text{Proj}_+\l(z^k-\alpha^k\nabla F(z^k)\r)-z^k\r)$$
where $\lambda^k$ and $\alpha^k$ are step sizes and $\proj_+$ is the projection on nonnegative orthant.  Two variants of the algorithm have been proposed. In the basic variant $\lambda^k=1$ and $\alpha^k$ is determined through a backtracking line search. The BB version finds $\lambda^k\in \l[0,1\r]$ using the technique due to Barzilai and Borwein \cite{BBstepsize}, then updates $\alpha^{k+1}$ accordingly. GPSR software is available at \cite{ GPSRs}. In our experiment we have used the BB version.
\item{\bf l1-ls \cite{l1lsp}}
l1-ls solves the same reformulation of BPDN problem as in GPSR through a truncated Newton interior point method. 
In l1-ls, preconditioned conjugate gradient (PCG) has been adopted for finding the search direction. Although forming the Newton system explicitly requires $A^tA$, the computational cost of each iteration of PCG is reduced to a matrix vector multiplication by the proper choice of preconditioner. The MATLAB code of this solver is available at \cite{l1lsS}.

\item {\bf FPC and FPC-AS \cite{FPCp1,FPCp2}}
Recall that first order optimality conditions imply that $\xs$ is the minimizer of a composite function if and only  if 
\beq \label{imroNR:FPeq}
\xs=\prox_{\alpha p}\l(\xs - \alpha \nabla f(\xs)\r).
\eeq

Equation \eqref{imroNR:FPeq} is called ``fixed point equation". ``Fixed Point Continuation" (FPC) aims to solve equation \eqref{imroNR:FPeq} through a proximal gradient method. The resulting algorithm has the following general scheme
\beq
x^{k+1} = \prox_{\alpha p} \l( x^k - \alpha \nabla f(x^k)\r).
\eeq
The developed theory on the convergence of this method suggests that the algorithm converges faster for larger values of $\lambda$. ``Continuation" strategy is essentially solving BPDN problem for a decreasing values of $\bar \lambda \rightarrow \lambda$ and warm starting the algorithm from the terminating solution corresponding to the previous value of $\bar \lambda$. 
FPC has later been extended to FPC-AS \cite{FPCp-as}. For each continuation interval, FPC-AS first solves the problem through FPC, then hard-thresholds the solution for nonzero entries. The $\|x\|_1$ is replaced  with $\sgn(x)^tx$, and the smaller smooth problem is minimized to attain the final solution. FPC and FPC\_AS  software are available at \cite{FPCs,FPC-ASs}. We have used FPC-AS in our experiment.

\item{\bf SpaRSA (Sparse Reconstruction by Separable Approximation) \cite{SpaRSAp}}
SpaRSA is a proximal gradient  framework for composite functions in which the nonsmooth part, $p(x)$, is separable. When $p(x)=\|x\|_1$ as is BPDN problem, SpaRSA reduces to an ISTA algorithm. The  Barzilai and Borwein \cite{BBstepsize} and a continuation scheme has been applied to enhance the performance of the algorithm. The MATLAB code is available at \cite{SpaRSAs}.

\item {\bf SPGL1 \cite{VFried:PP}}This method solves BP$_\eps$ formulation, \eqref{BP}, through solving a sequence of LASSO problems, \eqref{LASSO}. Each LASSO problem is solved using a spectral projected gradient method \cite{InexSpecProjGrad2}. In this technique, a single parameter function for LASSO problem is defined as $\phi(\tau)= \|Ax^\ast_\tau -b \|$. Using the dual information of the LASSO problem, one may recover derivative of $\phi$; hence Newton method is applied to find the root of $\phi(\tau)=\eps$, where $\eps$ is the parameter of BP$_\eps$ problem. At termination the solution of LASSO$_{\tau^\ast}$ coincidences with the solution of BP$_\eps$. 
The MATLAB package for SPGL1 is available at \cite{SPGL1s}.

\item {\bf TwIST\cite{TwISTp1}} Recall that ISTA algorithm as presented in section \ref{intro:SR} had the general form of $\shrk_{\lambda \alpha} \l(x^k- \alpha \nabla f(x^k)\r)$ or $(1-\beta)x^k + \beta \shrk_{\lambda \alpha} \l(x^k- \alpha \nabla f(x^k)\r)$ for $\beta=1$. TwIST is a two step ISTA technique in which each iterate depends on the last two previous iterates. The general form of TwIST is $(1-\alpha)x^{k-1} + (\alpha-\beta)x^k + \beta \shrk_{\lambda \alpha} \l(x^k- \alpha \nabla f(x^k)\r)$. Note that $\alpha$ and $\beta$ are not necessarily smaller than 1. For our experiment we have used the default setting on these parameters. The MATLAB package of this algorithm is available at \cite{TwISTs}.

\item {\bf FISTA\cite{BT:fista}} We presented FISTA in section \ref{intro:SR}. The algorithm has proposed by Beck and Teboulle in \cite{BT:fista} and by Nesterov in \cite{nesterov1983}. The algorithm is part of the TFOCS package \cite{TFOCSp} that is available at \cite{TFOCSs}.

\item {\bf ZeroSR1 \cite{BeckerPNM}} We presented this method and the similarities of this technique to IMRO earlier. As discussed, the Hessian approximation in this proximal quasi-Newton method is of the form ``diagonal plus a rank one matrix". This package can be used for minimizing composite functions in which $p(x)$ is separable; thus it is a suitable technique for BPDN problem. The software is in MATLAB and is available at \cite{ZeroSR1s}. The work per iteration of ZeroSR1 is almost identical to ours in the case of IMRO-1D. For IMRO-2D, one extra $A^t$ call is required. For this reason, our experiments compare the number of multiplications by $A/A^t$ rather than the number of iterations.
 
\item {\bf PNOPT \cite{ProxNMcompfun}} PNOPT is a MATLAB implementation of the proximal Newton-type methods for minimizing composite functions. The Hessian approximation in the proximal quasi-Newton methods are based on BFGS and L-BFGS. In our experiment we have used L-BFGS as it performed considerably better than BFGS on our test cases. This package is available at \cite{PNOPTs}.
\end{itemize}

In both variants of IMRO, we took $x^{k+1}$ as the minimizer of the approximation model, i.e., no line search has been employed.
One of the key questions regarding the implementation of IMRO-1D is what direction $v$ should be used. In our experiment, we have used the previous direction for $v^k$, i.e. $v^k=x^k-x^{k-1}$. The choice of $v^k$, however, needs further study. Our measurement on the computational cost of each algorithm is the number of matrix-vector multiplications, i.e., the number of calls to $A$ or $A^t$. We plot the residual of the solution (i.e., $\|x-\xs\|$) with respect to $A$/$A^t$ calls, where $\xs$ is obtained from our test generator.

%%%%%
Our test cases come from L1TestPack \cite{l1testpack} and Sparco \cite{sparco}. We first compare IMRO with available software for solving BPDN formulation on L1TestPack \cite{l1testpack} generated instance. We, then, test IMRO and couple of other state-of-the-art solvers in sparse signal recovery on a number of examples from Sparco \cite{sparco}. Table \ref{imroNR:infoNonOrth} summarizes the information on the test cases used for our computational experiment. For the L1TestPack generated instances, ``Ent.\ type" stands for the type of entries in matrix $A$ or vector $\xs$, the optimal solution of the problem. ``dynamic 3" gives random entries with a dynamic range of $10^3$. It is often believed that instances with higher dynamic range are harder to solve. According to our experiment, $\text{cond}(A)$ is another key factor that influences the performance of some of the solvers. 

Figures \ref{imro:accgraph} and \ref{fig:errsparco} visualize the residual of the solution with respect to the number of $A$ or $A^t$ calls for the proximal gradient methods. While FIMRO performs slightly better than IMRO-1D, our experiment suggests that IMRO-2D outperforms both FIMRO and IMRO-1D, and competes favorably with other solvers. 

\begin{table}[h]
\caption{Information on Test Cases}
\centering
\begin{tabular}{lcccccc}
\\[-.3cm]
\hline
& & &   cond(A) & Ent. type of $A$ & Ent. type of $x$ &$\lambda$ \\
\hline
\multirow{6}{*}{{ \begin{turn}{90}L1TestPack\quad \quad \end{turn}}}&
\multirow{6}{*}{{ \begin{turn}{90}$A \in \mathbb R^{2500 \times 10000}$ \hspace{0.1in} \end{turn}}}\\
& & Ins 1   & $O(1)$ & Gaussian& Gaussian & 0.5 \\
& &Ins 2  & $O(1)$ & Gaussian& dynamic 3 & 0.5 \\
& &Ins 3 & $O(1)$ & Gaussian& Gaussian & 0.1 \\
& &Ins 4  & $O(1)$ & Gaussian& dynamic 3 & 0.1 \\
& &Ins 5  & $O(10^3)$ & Gaussian& Gaussian & 0.1 \\
& &Ins 6  & $O(10^3)$ & Gaussian& dynamic 3 & 0.1 \\
\hline
\hline
& & ID &  m & n & \multicolumn{1}{c}{ Operator} & $\lambda$ \\
\hline
\multicolumn{2}{c}{\multirow{4}{*}{{\begin{turn}{90} \large Sparco\hspace{0.15 in} \end{turn}}}}\\
%\\
& &5& 300 & 2048 & \multicolumn{1}{c}{Gaussian,DCT}&0.1
\\
& &9& 128 &	128 & \multicolumn{1}{c}{Heaviside}&0.1
\\
& &10&1024&1024 & \multicolumn{1}{c}{Heaviside}&0.1
\\
& & 903& 1024 &1024 &\multicolumn{1}{c}{ 1D Convolution }& 0.1
\\
\hline
\end{tabular}
\label{imroNR:infoNonOrth}
\end{table}

%%%%%%%%%%%%
%%%%%%%%%%%%

\begin{figure}[t]
\centering
\caption{Accuracy of the Solution for BPDN Solvers on L1TestPack Generated Instance }
\label{imro:accgraph}
\subfigure[Ins 1]{
\includegraphics[width=6cm, height=6cm]{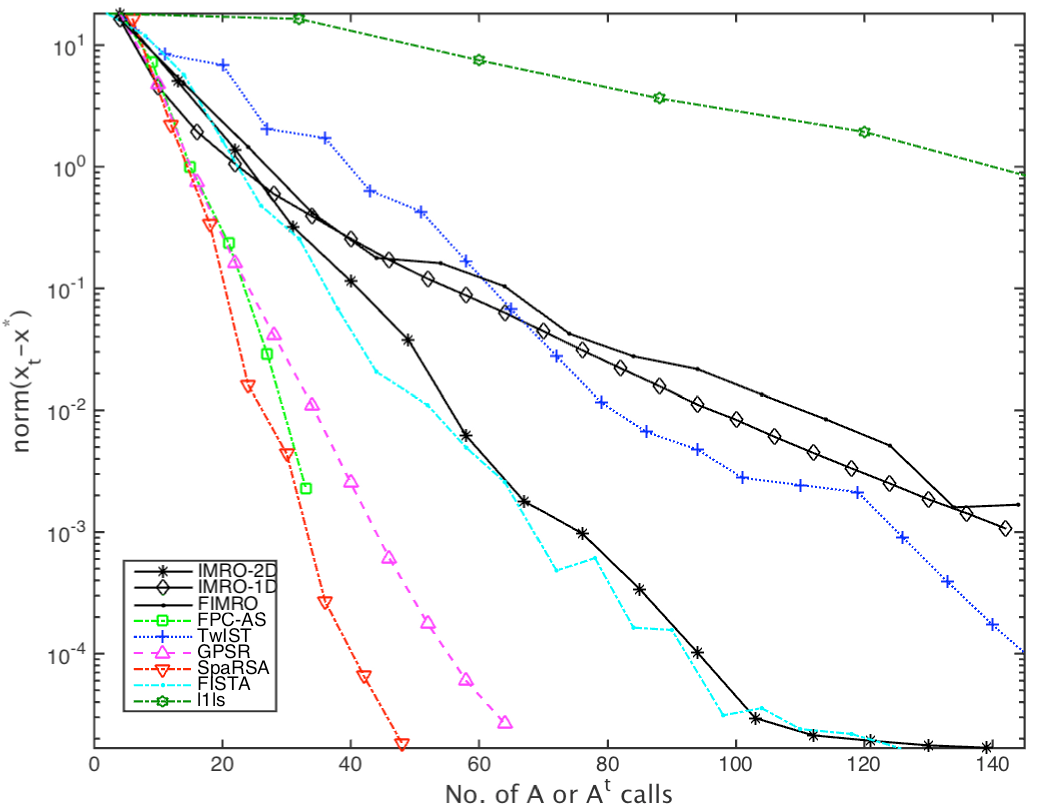}
}
\subfigure[Ins 2]{
\includegraphics[width=6cm, height=6cm]{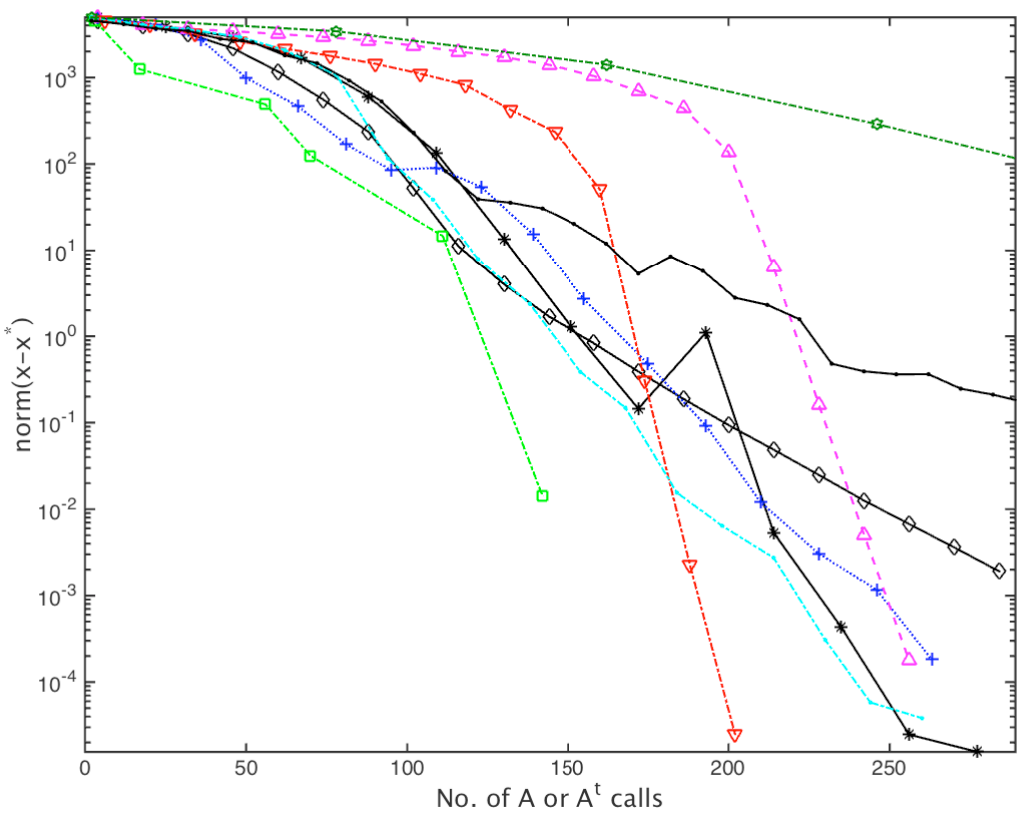}
}

\subfigure[Ins 3]{
\includegraphics[width=6cm, height=6cm]{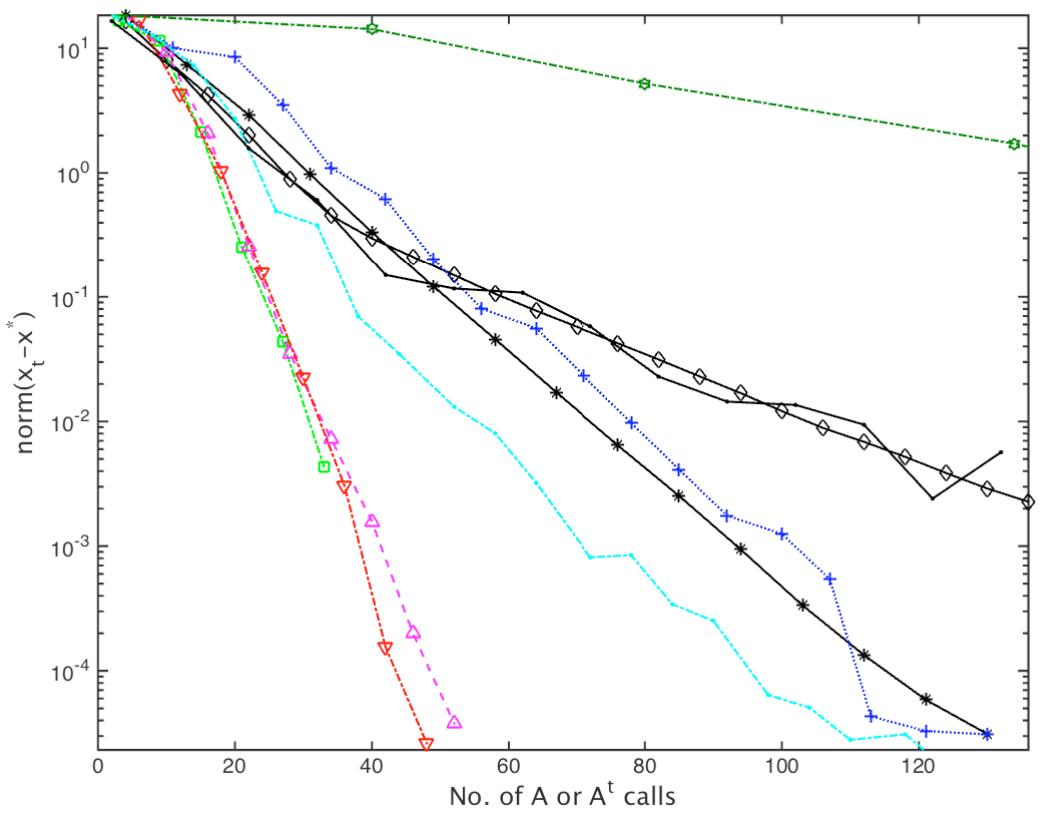}
}
\subfigure[Ins 4]{
\includegraphics[width=6cm, height=6cm]{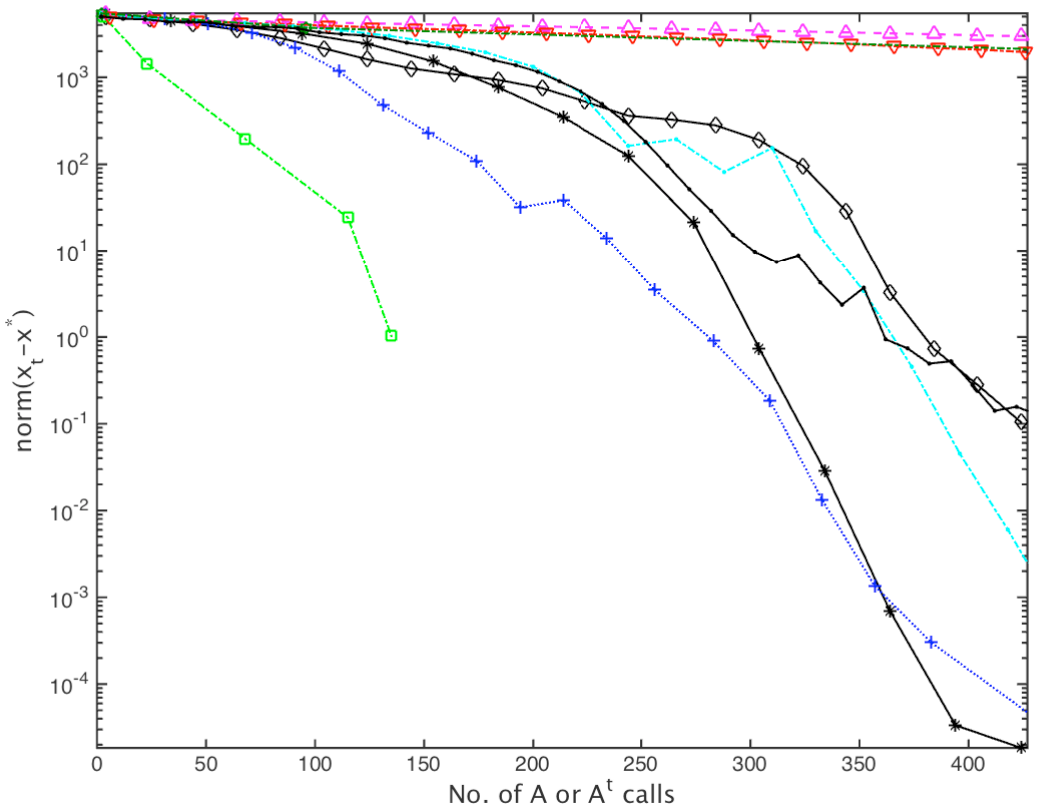}
}
\subfigure[Ins 5]{
\includegraphics[width=6cm, height=6cm]{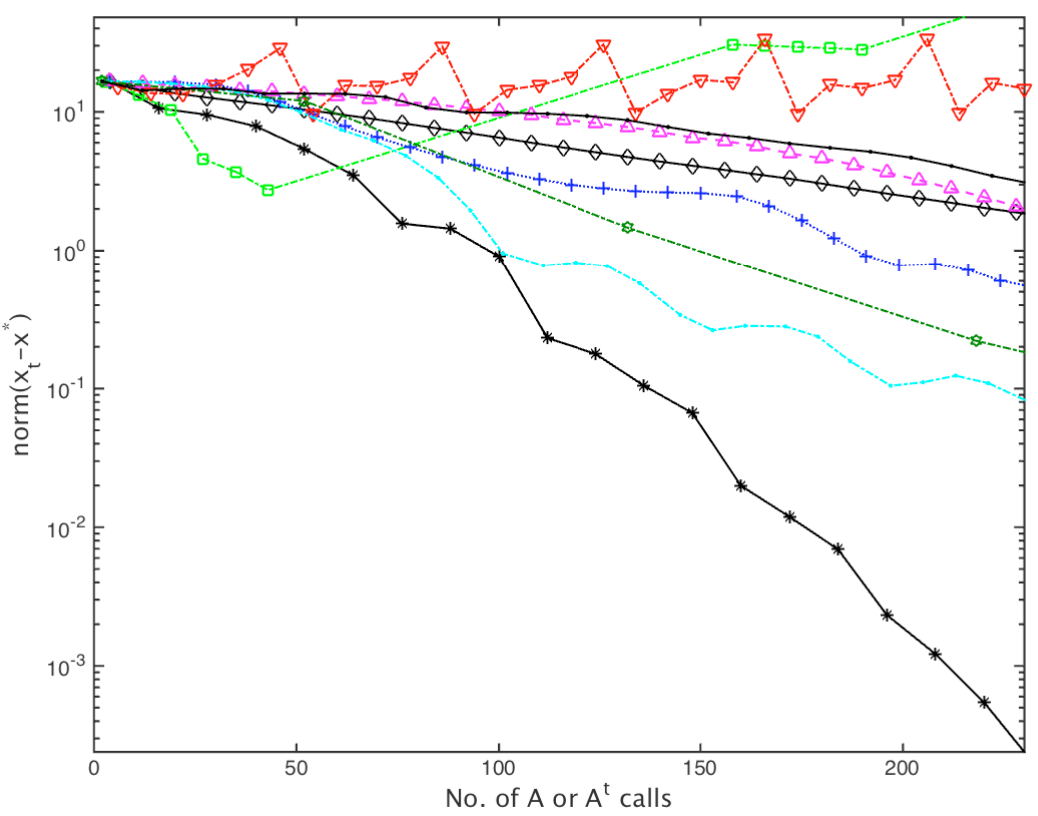}
}
\subfigure[Ins 6]{
\includegraphics[width=6cm, height=6cm]{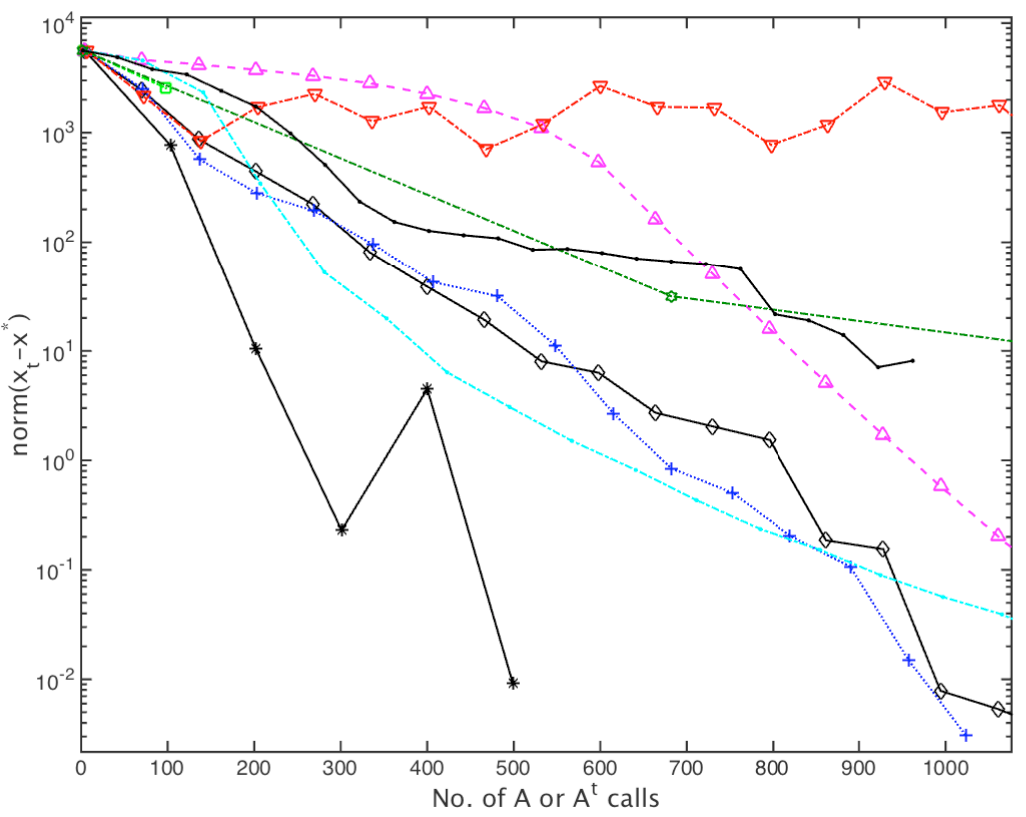}
}
\end{figure}

\begin{figure}[]
\centering
\caption{Accuracy of the Solution for Sparse Recovery Solvers on Sparco Test Cases }
\label{fig:errsparco}
\subfigure[Sparco(5)]{
\includegraphics[width=6cm, height=6cm]{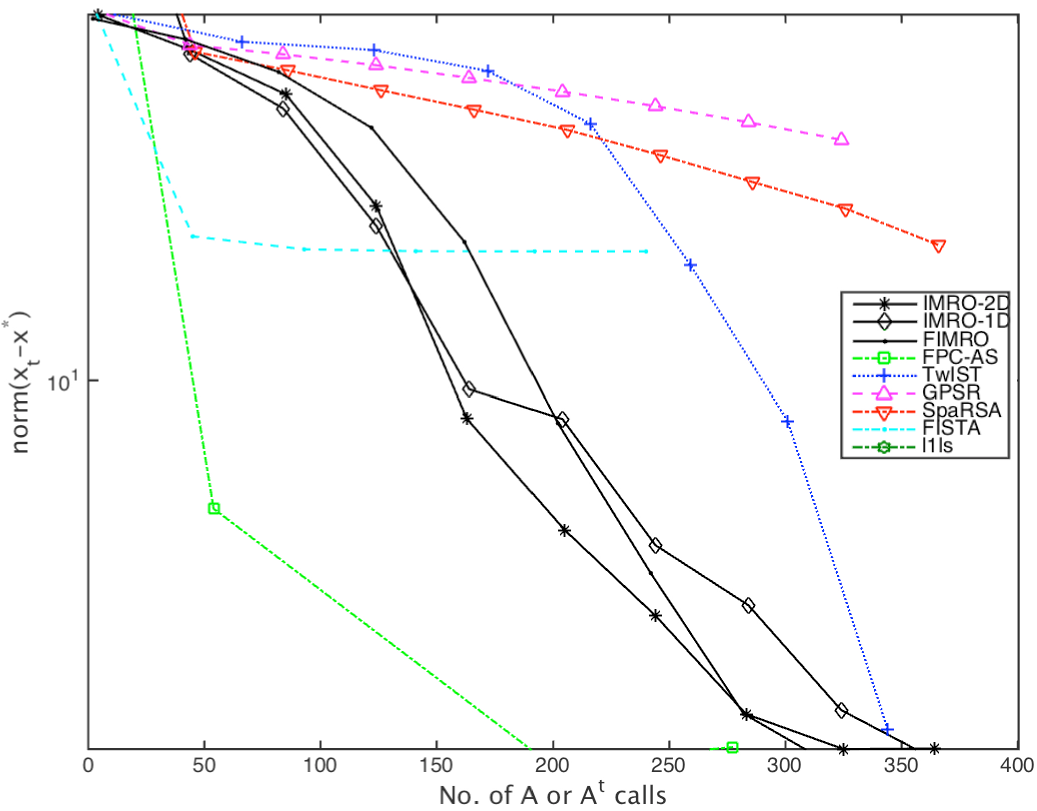}
}
\subfigure[Sparco(9)]{
\includegraphics[width=6cm, height=6cm]{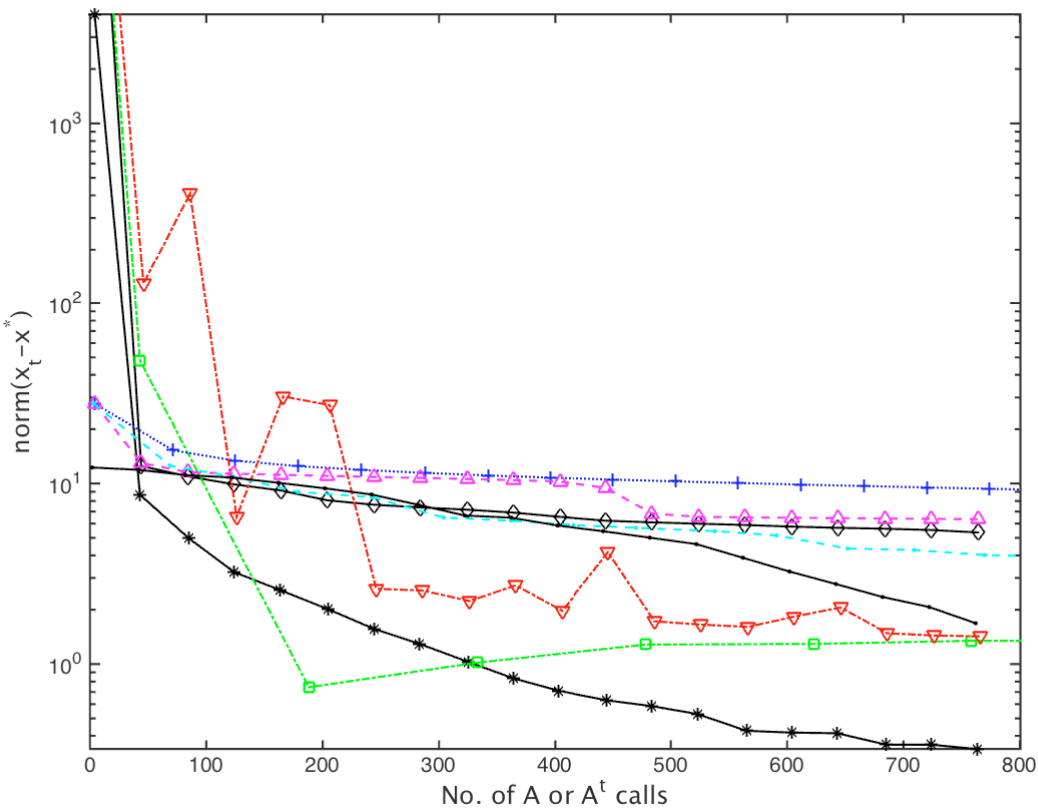}
}

\subfigure[Sparco(10)]{
\includegraphics[width=6cm, height=6cm]{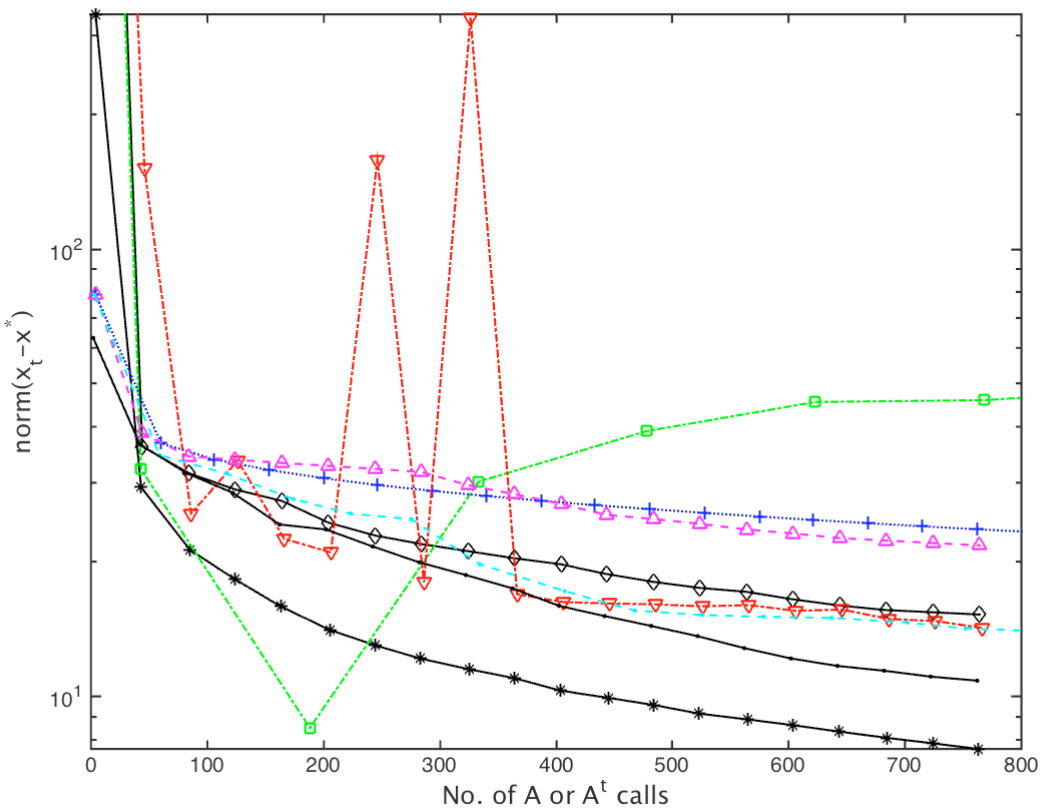}
}
\subfigure[Sparco(903)]{
\includegraphics[width=6cm, height=6cm]{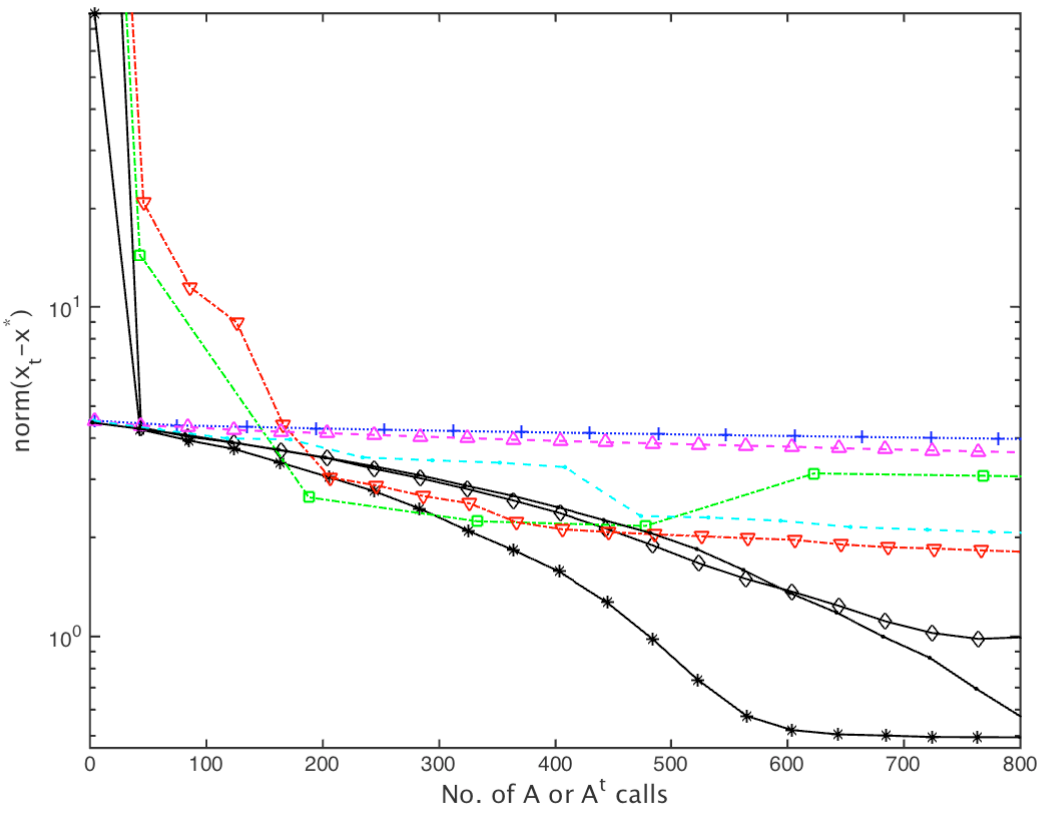}
}
\end{figure}
%%%%%%%%%%%

In figures \ref{imro:PQNgraph1} and \ref{imro:PQNsparco}, we compare IMRO-2D with other proximal quasi-Newton techniques, i.e., ZeroSR1 \cite{ZeroSR1s}  and PNOPT \cite{PNOPTs} on our test cases. The default setting for the number of stored previous iterates in the L-BFGS of PNOPT is 50, which is substantially larger than what is typically used for L-BFGS \cite{nocedalwrightbook}.  We hypothesize that storage of 50 vectors is impractical for ``big-data'' applications in which each iterate has millions of entries.  In addition, it is not clear how to compare the PNOPT running time with 50 stored vectors against other codes because, unlike the other codes, the segments of the PNOPT code that compute inner products and saxpys require a substantial amount of CPU time. In addition, the performance of the solver was not improved for 50 stored vectors on our test cases. For these reasons, we tested PNOPT with the memory set to 3 instead of 50. 
\begin{figure}[t]
\centering
\caption{Accuracy of the Solution for the Proximal Quasi-Newton Techniques on L1TestPack Generated Instances}
\label{imro:PQNgraph1}
\subfigure[Ins 1]{
\includegraphics[width=6cm, height=6cm]{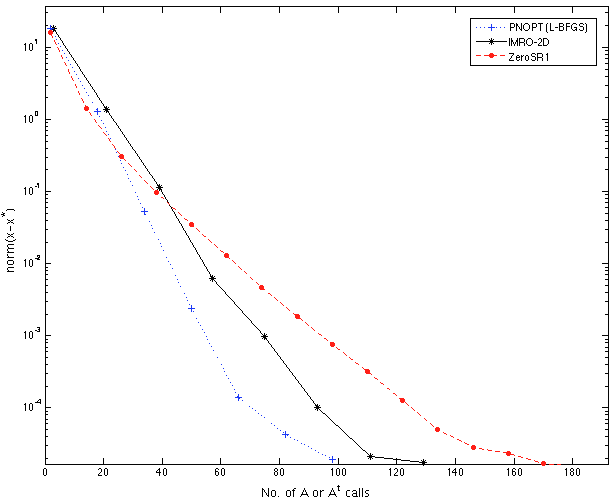}
}
\subfigure[Ins 2]{
\includegraphics[width=6cm, height=6cm]{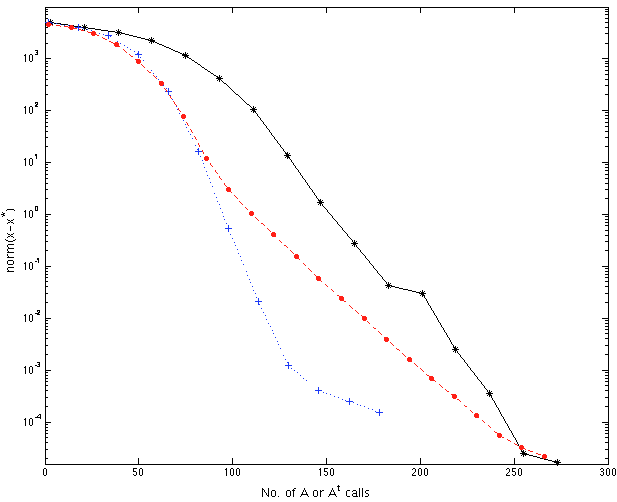}
}

\subfigure[Ins 3]{
\includegraphics[width=6cm, height=6cm]{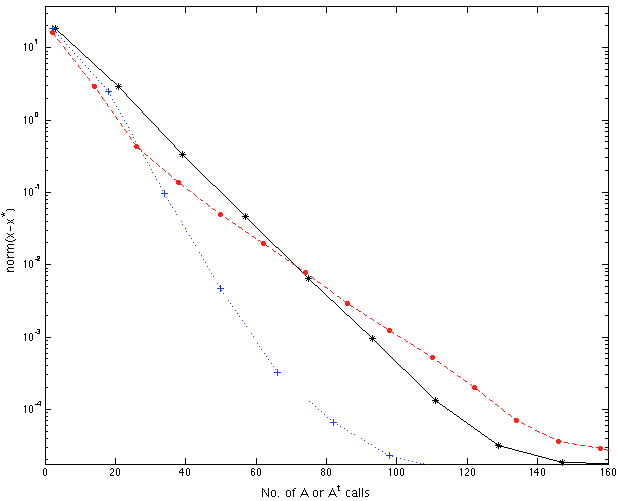}
}
\subfigure[Ins 4]{
\includegraphics[width=6cm, height=6cm]{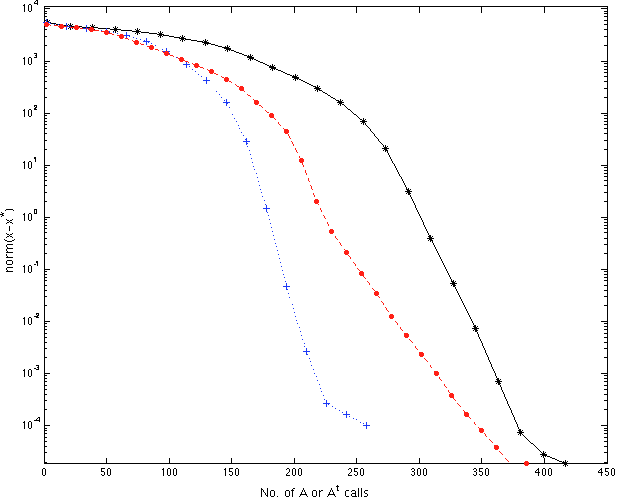}
}
\subfigure[Ins 5]{
\includegraphics[width=6cm, height=6cm]{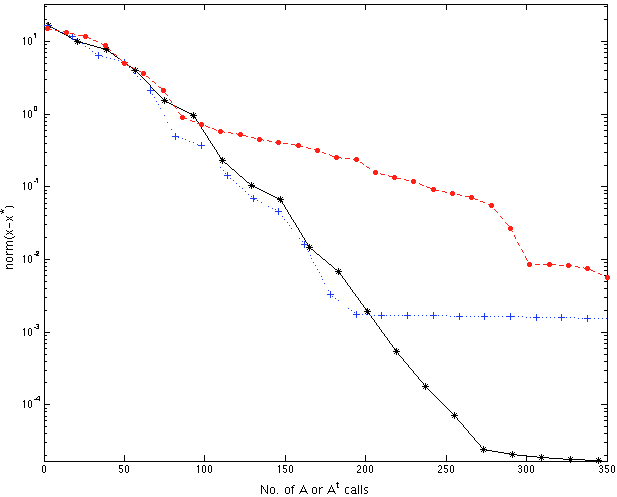}
}
\subfigure[Ins 6]{
\includegraphics[width=6cm, height=6cm]{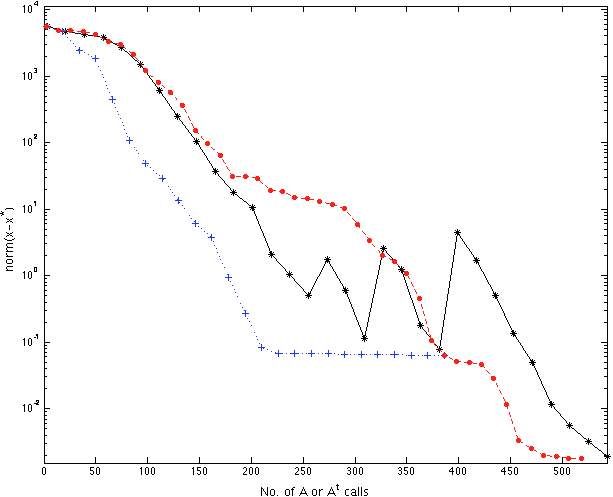}
}
\end{figure}

\begin{figure}[]
\centering
\caption{Accuracy of the Solution for Proximal Quasi-Newton Techniques on Sparco Test Cases }
\label{imro:PQNsparco}
\subfigure[Sparco(5)]{
\includegraphics[width=6cm, height=6cm]{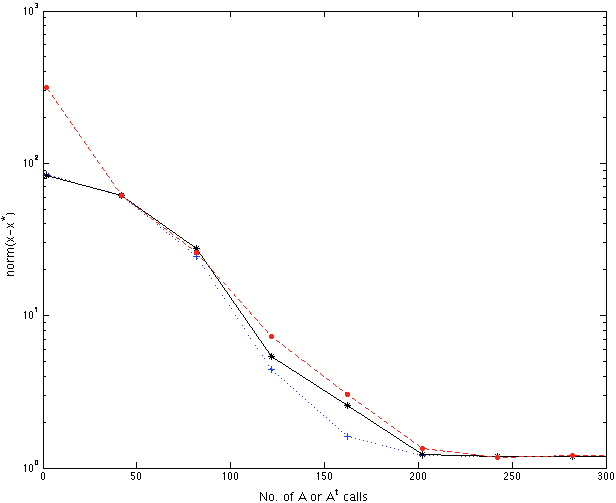}
}
\subfigure[Sparco(9)]{
\includegraphics[width=6cm, height=6cm]{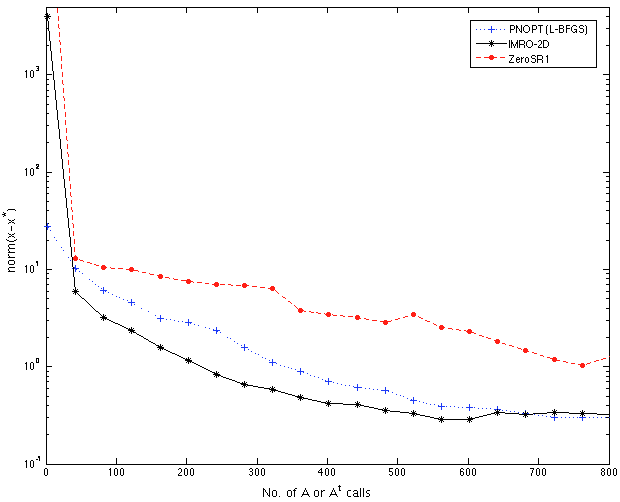}
}

\subfigure[Sparco(10)]{
\includegraphics[width=6cm, height=6cm]{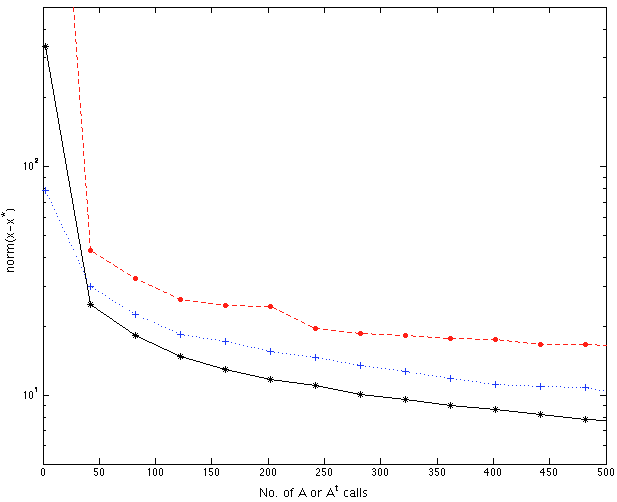}
}
\subfigure[Sparco(903)]{
\includegraphics[width=6cm, height=6cm]{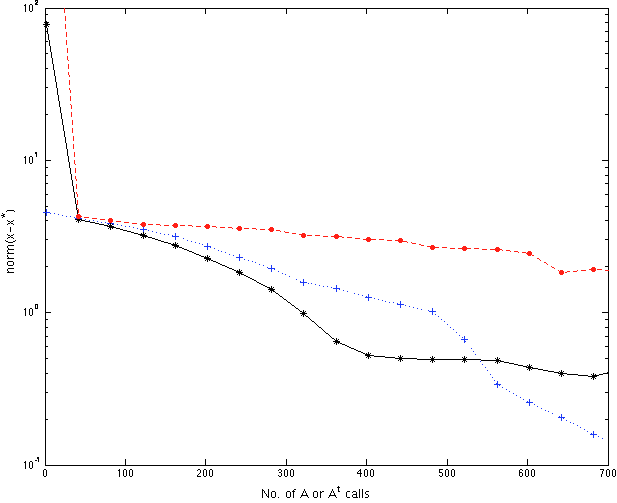}
}
\end{figure}

\section{Conclusion}\label{sec:conclusion}
We presented a proximal quasi-Newton method for solving the $\ell_1$-regularized least square problem. The approximation Hessian matrix in the suggested scheme has the format of identity minus rank one (IMRO) which allows us to compute the proximal point effectively. Two variants of this technique are proposed; in IMRO-1D the approximation model matches the function on a one-dimensional space, and in IMRO-2D it matches the function on a two-dimensional space. Our computational experiments have shown promising results. IMRO-2D, in particular, outperformed other state-of-the-art solvers in many of our test cases. Although this paper focuses on solving BPDN formulation, there are generalizations possible. As noted by Becker and Fadili \cite{BeckerPNM}, the computation of the proximal point is generalized to other regularizers besides $\|x\|_1$. For the smooth part, we can approximate the function rather than matching it exactly on a specific subspace. For the quadratic approximation to the smooth part the gradient and the Lipschitz constant of the gradient is needed for IMRO-1D, or the local curvature of the function in two directions for IMRO-2D, which can be obtained efficiently via automatic differentiation \cite{mythesis}. These extensions of our method could be the subject of future work. 

An accelerated variant of IMRO, named FIMRO, was also proposed. Despite theoretical advantages of FIMRO, we did not observe significant practical improvement in our experiment. The possible directions (one dimensional space) in IMRO-1D needs further study. Other possible directions to pursue are to adapt a suitable line search and a continuation scheme  for IMRO. Although in theory the convergence rate of IMRO does not depend on the regularization parameter, a continuation scheme may enhance the performance of IMRO in practice.

\clearpage

\bibliographystyle{plainurl}
\bibliography{RefIMROp}

\end{document}